\documentclass[11pt]{article}

\usepackage[margin=1in]{geometry}
\usepackage{graphicx}
\usepackage{enumerate}
\usepackage{amsmath}
\usepackage{amssymb}
\usepackage{hyperref}
\usepackage{xcolor}
\usepackage{amsthm}
\usepackage[round]{natbib}
\usepackage{graphicx}
\usepackage{subcaption}
\usepackage{booktabs}
\usepackage{float}
\DeclareMathOperator*{\argmax}{arg\,max}

\allowdisplaybreaks

\newtheorem{theorem}{Theorem}
\newtheorem{proposition}[theorem]{Proposition}
\newtheorem{lemma}[theorem]{Lemma}
\newtheorem{definition}[theorem]{Definition}
\newtheorem{assumption}[theorem]{Assumption}
\newtheorem{corollary}[theorem]{Corollary}
\newtheorem{remark}[theorem]{Remark}
\newcommand{\FIGURE}[3]{\centering #1\caption{#2}}

\newcommand{\tr}{\mathrm{Tr}}

\newcommand{\cW}{\mathcal W}
\newcommand{\cG}{\mathcal G}

\newcommand{\cL}{\mathcal L}

\newcommand{\bH}{\mathbf{H}}
\newcommand{\bW}{\mathbf{W}}
\newcommand{\R}{\mathbb R}
\newcommand{\E}{\mathbb E}

\begin{document}
\title{Continuous Policy and Value Iteration for Stochastic Control Problems with Inhomogeneous Synchronous Coupling}

\author{
Qi Feng\thanks{\noindent Department of
Mathematics, Florida State University, Tallahassee, 32306; email: qfeng2@fsu.edu. This author is partially supported by the National Science Foundation under grant \#DMS-2420029.}
\and Renjie Shao\thanks{\noindent Department of
Mathematics, Florida State University, Tallahassee, 32306; email: rs25a@fsu.edu.}
\and Gu Wang\thanks{ \noindent Department of
Mathematical Sciences, Worcester Polytechnic Institute, Worcester, MA 01609;
email: gwang2@wpi.edu. This author is partially supported
by National Science Foundation grant \#DMS - 2206282 and John F. Manning, Jr. '80 and Catherine C. Leser A\&S Faculty Research Grant from Worcester Polytechnic Institute.} 
}
\date{}
\maketitle
\begin{abstract}
We propose a continuous policy-value iteration framework for infinite-horizon stochastic control problems, in which the policy evolves through Langevin-type dynamics driven by the gradient of the in-homogeneous Hamiltonian, while the value function is updated simultaneously and continuously along the iteration, avoiding directly solving its corresponding  PDE at each step. The iteration scheme enjoys policy improvement under suitable regularity and admissibility assumptions and we show the exponential contraction of the joint policy–value dynamics in Wasserstein-2 distance under the Lipschitz condition in expectation via the inhomogeneous synchronous coupling and verify these conditions for linear–quadratic problems. Numerical experiments on linear–quadratic models and the optimal consumption–investment problems in an incomplete market factor model illustrate the policy improvement, convergence toward optimal policies, and the effect of decaying exploration noise.
\end{abstract}

{$\mathbf{Keywords}$.} Hamilton-Jacobi-Bellman equation; Reinforcement learning; Policy gradient ascent; Stochastic control; Synchronous coupling; Langevin dynamics.


\section{Introduction}

Inspired by the (stochastic) gradient descent method for optimization problems, policy iteration algorithms for stochastic control problems have gained significant attention in recent years \citet{puterman1979convergence,puterman1981convergence,bertsekas2015value,jacka2017policy,kerimkulov2020exponential,kerimkulov2021modified}. These iteration schemes provide approximations, which converge to the optimal control and value function and early results focus on discrete time settings \citet{sutton1998reinforcement,bertsekas2012dynamic,szepesvari2022algorithms,bertsekas2011dynamic}. \citet{wang2020reinforcement,sethi2024modified,zhou2023policy,ma2024convergence} extend the discussion to a continuous-time setting, by applying various techniques, e.g. dynamic programming principle, stochastic maximum principle, and policy gradient. 

In the relaxed control setting, \citet{feng2026continuous} recently introduces a continuous iteration method, in which both the policy and the value function are described by Langevin dynamics, and exponential convergence to the optimal solution is shown under the monotonicity conditions for the value iteration. \citet{feng2026score} developed a numerical scheme based on the score-Hamiltonian projection, which allows implementing this iteration scheme  with high dimensional state and control. In this paper, we propose a continuous iteration for stochastic control problems in the classical setting. The monotonicity conditions are easily fail in this set up and the control iteration following the Langevin dynamics makes the value function random -- a major difference from the relaxed control problem where the policy, as the law of the Langevin dynamics, is deterministic. 

Based on gradient ascent argument, Lemma \ref{lemma:policy-improve-c} shows that value functions are improved along the policy dynamics as $\tau\rightarrow \infty$ in the classical control framework with temperature coefficient $\lambda=0$ in the control iteration dynamics, while for $\lambda \neq 0$, the improvement is in average sense. In the case of non-concave Hamiltonian, Proposition \ref{prop:control-update-classic} prescribes a manual update of the control to jump out of the local maximum for the deterministic iteration. For the convergence of general iteration dynamics with non-zero temperature,  we adapt the idea of synchronous coupling techniques from Langevin dynamic \citet{Dalalyan2020} \citet{EGZ19} to establish exponential contraction of the joint policy-value dynamics under coupling bounds and small-gain condition \eqref{eq:smallgain}. These newly proposed conditions replace the monotonicity condition in \citet{feng2026continuous} by assumptions on coupling bounds, ensuring that coupling contraction dominates the mutual interaction between policy and value dynamics, inspired by \citet{smallgain1996}. To retrieve the non-random optimal control strategy, we also apply the annealing-based continuous-time Langevin dynamics MCMC algorithms \citet{kirkpatrick1983optimization,vcerny1985thermodynamical,geman1986diffusions,gelfand1990sampling,tang2021simulated,feng2024fisher},
by scheduling the temperature $\lambda(\tau)\rightarrow 0$, as $\tau\rightarrow \infty$. 

Our approach builds a bridge between solving stochastic control problems and applying non-convex optimization techniques inspired by annealing-based continuous-time Langevin dynamics MCMC algorithms. By incorporating our continuous dynamics in iteration, we extend the policy improvement algorithm to policy gradient ascent/descent methods, which have the potential to incorporate stochastic gradient ascent/descent algorithms \citet{welling2011bayesian,zhang2017hitting,raginsky2017non,dalalyan2020sampling} for high-dimensional problems. Compared to the continuous dynamics with gradient flow in \citet{vsivska2024gradient,ZZH2025} for finite horizon problems, we take advantage of the infinite horizon setting and only update the policy locally and infinitesimally at every step in the iteration, instead of its whole path at every point in time. Since the value function is updated simultaneously with the control, our continuous policy dynamics is inhomogeneous in nature, i.e. the objective function in the Langevin dynamics moves as iteration proceeds. 

We test our iteration method in an optimal consumption-investment problem with power utility in a factor model. In the general setting with incomplete market and beyond the case of logarithmic utility, the associated Hamilton-Jacobi-Bellman (HJB) equation does not have closed-form solutions and there has been vast literature on approximate solutions, using e.g. market completion \citet{he1991consumption,karatzas1991martingale}, duality argument \citet{GW2020,bick2013solving} or converging sub- and super-solutions \citet{gutekunst2026optimalinvestmentconsumptionstochastic}. All of these methods require delicate mathematical treatment and lose the tractability or its accuracy even beyond the one-dimension case. Our method provide easily implementable alternative, which only requires elementary calculations through the iteration, and provides accurate approximations with exponential convergence.

This paper is organized as follows. In Section \ref{sec: classic}, we define the infinite horizon stochastic control problem and its associated Hamilton-Jacobi-Bellman (HJB) equation, which motivates the iteration scheme. We show that with constant temperature in the Langevin dynamics for the control iteration, the corresponding value increases in average. In Section \ref{sec: convergence}, we analyze the convergence of the scheme via synchronous coupling, under the Lipschitz conditions in expectation along the iteration, and we verify the conditions for the linear-quadratic model. To retrieve the non-random optimal solution, we also use the annealing technique and decrease the temperature along the dynamics, which allows us to the show exponential convergence in $L^2$. Section \ref{sec: numerics} presents numerical experiments on a linear-quadratic model and the example in mathematical finance.

\section{Infinite Horizon Stochastic Control Problem}\label{sec: classic}
Consider a controlled diffusion process in \(\mathbb{R}^d\), which follows
\begin{equation*}
	dX^u_t = b(X^u_t,u_t)dt +\sigma(X^u_t,u_t)dW_t, \quad X_0 = x,
\end{equation*}
where $b: \mathbb R^d\times \mathbb R^n \rightarrow \mathbb R^d$, $\sigma: \mathbb R^d\times \mathbb R^n \rightarrow \mathbb R^{d\times k}$ satisfying Assumption \ref{classical-wellposed} below, and $W$ is a $d$-dimensional Brownian motion. For the ease of notation, we drop the superscript $u$ in the following discussion if no ambiguity arises.
The agent's goal is to choose the optimal control process from the admissible set, which maximizes the cumulative discounted running payoff $f(X_t,u_t)$, i.e., $V(x) =  \sup\limits_{u\in\mathcal A(x)} J(x,u)$, where
\begin{equation*}
	J(x,u) =  \mathbb E\left[\int_0^\infty e^{-\beta t}f(X_t,u_t)dt\right],
\end{equation*}
and the admissible set \(\mathcal{A}\) of the control $u$ is defined as
\begin{definition}
For any initial state $x\in\mathbb R^d$, the admissible set $\mathcal A(x)$ is a collection of strategies $u = \{u_t \in \mathbb R^n, t\geq 0\}$, progressively measurable to $\{\mathcal F^W_t\}_{t\geq 0}$, the filtration generated by \(W\), and satisfy 
\begin{enumerate}
	\item[(i)] $\mathbb E\left[\int_0^T \left(|b(X_t,u_t)|^2 + \tr[\sigma\sigma^\top(X_t,u_t)]\right) dt \right]<\infty$, for each $T > 0$.
	\item[(ii)] $J(x,u) < \infty$.
\end{enumerate}
\end{definition}
We make the following assumption on the model parameters and the well-posedness of the optimization problem.
\begin{assumption}\label{classical-wellposed}
 \begin{enumerate}
        \item[(i)] $\beta>0$, \(f, b\) and \(\sigma \sigma^\top\) are \(C^2\) in both \(x\) and \(u\), $b$ and $\sigma$ are bounded and all the first-order and second-order derivatives are uniformly bounded; 
        \item[(ii)] \(\sigma \sigma^\top\) is unifromly positive definite, i.e. \(\sigma \sigma^\top (x,u)\succeq cI\)  for some constant \(c >0\);
        \item [(iii)] $V(x) < \infty$.
    \end{enumerate}
\end{assumption}
The HJB equation associated to this control problem is
\begin{align}
	&-\beta v(x) + \sup_{u}\left(f(x,u) + \frac{1}{2}\tr[\sigma\sigma^\top(x,u):v_{xx}(x)] + b^\top(x,u)v_x(x)\right)\\
	=& -\beta v(x)  +\sup_{u} \mathbf{H}(x,u,v_x,v_{xx}) =0 \label{eq:HJB},
\end{align}
where $\mathbf{H}  :=\mathbf{H}(x,u,v_x,v_{xx}) =f(x,u) + \frac{1}{2}\tr[\sigma\sigma^\top(x,u):v_{xx}(x)] + b^\top(x,u)v_x(x)$, $v_x$ and $v_{xx}$ are the gradient and Hessian of $v$ with respect to $x$.
If $\mathbf{H}(x,u,v_x,v_{xx})$ is strictly concave in $u$, then there is a unique maximizer $\hat u$ pf $\bH$, which solves the equation $\nabla_u \mathbf{H}(x,u,v_x,v_{xx}) = 0$, and we denote as $U(x,v_{x},v_{xx})$. If $\hat v = V$, which is supposed to solve the HJB equation, then $\hat u = U(x,\hat v_x,\hat v_{xx})$ is the optimal control. 

Generally, especially with controlled diffusion coefficient, the HJB equation is fully non-linear and rarely obtains closed form solutions, let alone the optimal control out of the first order condition. Without the concavity of $\mathbf H$, even the first order condition may not yield the optimal control. In this paper, we propose an iteration scheme for both the control $u$ and its corresponding value $v$, which converge to the optimal solution: for each state $x\in\mathbb R^d$, given initial value $u^x_0$, let the control at state $x$ following the Langevin dynamics
	\begin{equation}\label{eq:u-dynamics-c}
		du^x_\tau = \nabla_u \mathbf{H}(x,u^x_\tau,v^\tau_x(0,x),v^\tau_{xx}(0,x))d\tau+ \sqrt{2\lambda}dB_\tau,
	\end{equation}
where $B$ is an $n$-Brownian motion independent of $W$, and we use $u_\tau^x$ as the approximate of the optimal control at $x$. The above dynamics depend on $v^\tau$, the result of the value iteration, which means that the control and the value of the optimization problem are updated simultaneously along the iteration. In the following we will denote the strategy $u^x_\tau$ as a function of state $x$, $u_\tau(x)$, to emphasize the dependence, and omit the argument $x$ if no ambiguity arises. We define $v^\tau$ in the following manner: 
\begin{definition}
For each $\tau\geq 0$, given $\{u_s(x),0\leq s\leq \tau\}$ for every $x$, define $u^\tau = \{u^\tau_t(X^\tau_t)\}_{0\leq t<\infty}$ as
\begin{equation}\label{eq:u-reverse}
u^\tau_t(x) = \begin{cases}
u_{\tau-t}(x), &0\leq t \leq \tau;\\
u_0(x), & t > \tau.
\end{cases}
\end{equation}
Denote the corresponding state process as $X^{\tau}$, satisfying
\begin{equation*}
	dX^\tau_t = b(X^\tau_t,u^\tau_t(X^\tau_t))dt +\sigma(X^\tau_t,u^\tau_t(X^\tau_t))dW_t,
\end{equation*}
and the corresponding value function, given a realization of strategy $u$ up to $\tau$, is
\begin{equation}
	v^\tau(t,x) = \mathbb E\left[\int_t^\infty e^{-\beta(s-t)}f(X^\tau_s,u^\tau_s(X^\tau_s))ds\,\Big|\,X^\tau_t = x, \mathcal F^B_\tau\right],\label{definition:v-tau-c}
\end{equation}
\end{definition}
where $\{\mathcal F^B_t\}_{t\geq 0}$ is the filtration generated by $B$. Notice that at each $\tau$, though strategy $u^\tau$ defined above is a strategy for an infinite horizon stochastic control problem, it only relies on $u_s$, $0\leq s\leq \tau$, i.e. the iteration results up to the current step, and not the future steps. Also, even though $u_\tau$ is a diffusion process driven by Brownian motion $B$, in the above definition of $v^\tau$, given $\mathcal F^B_\tau$, each $u^\tau_t$ is a deterministic function of the state $x$. Thus $u^\tau = \{u^\tau_t(X^\tau_t)\}_{0\leq t<\infty}$ is progressively measurable with respect to $\{\mathcal F^B_t, t\geq 0\}$. We make the following assumptions on $u^\tau$, $X^\tau$ and $v^\tau$.
\begin{assumption}\label{assumption-T-c}
	Given initial state $X_0 = x$, $\tau \geq 0$, and sufficiently small $h>0$,
\begin{enumerate}
	\item[(i)] For each realization of $\{u_s,0\leq s\leq \tau\}$, $u^\tau \in\mathcal A(x)$ and $v^\tau(t,x) \in C^{1,4}$.
	\item[(ii)]  $\mathbb E\left[\int_0^{\tau +h} \|v^{\tau}_x(s,X^{\tau+h}_s)\|^2ds\right]< \infty$.
    \item[(iii)] $\mathbb E\left[\left(|v^{\tau-s+l}|+\|v^{\tau-s+l}_x\| + \|v^{\tau-s+l}_{xx}\|+\|\nabla_u\ln \pi_{\tau-s+l}\|^2\right)(X^{\tau+h}_s,u_{\tau-s+l})\right]$ is bounded in $0\leq l \leq h$, for every $0\leq s\leq \tau+h$.
\end{enumerate}
\end{assumption}
Notice that $u^\tau$ defined above is time-inhomogeneous, and thus $v^\tau$ is a function of both time and state. However, we only use $v(0,x)$ as the approximate of the value function, and use it in the Langevin dynamics of $u_\tau$. Thus in following discussion, we mainly focus on the updates of $v^\tau(0,x)$ for a given $x$, and for the ease of notation, we drop the time argument and denote it as $v^\tau(x)$, if no ambiguity arises. The following lemma derive the continuous dynamics of the value $v^\tau$ in the iteration direction.
\begin{lemma}\label{dynamic v classical}
    The value function $v^{\tau}(x)$, and its first and second order derivatives $v^{\tau}_x$, $v^{\tau}_{xx}$ satisfy the following dynamics,
    \begin{align}\label{eq:v-dynamics-c}
dv^\tau(x) =& -\beta v^\tau(x) + \left(f +b^{\top}v^\tau_x + \frac{1}{2}\tr[\sigma\sigma^{\top}:v^\tau_{xx}]\right)(x,u_\tau)d\tau;\\
dv^\tau_x(x) =& -\beta v^\tau_x(x) +\left(\nabla_x\Big( f+b^{\top} v^\tau_x + \frac{1}{2}\tr[\sigma\sigma^{\top}:v^\tau_{xx}]\Big)\right)(x,u_\tau) d\tau;\nonumber\\
dv^\tau_{xx}(x) =& -\beta v^\tau_{xx}(x)+\left(\nabla^2_{xx}\Big( f(x,u_\tau) +b^{\top} v^\tau_x + \frac{1}{2}\tr[\sigma\sigma^{\top}:v^\tau_{xx}]\Big)\right)(x,u_\tau)d\tau.\nonumber
\end{align}
\end{lemma}

\begin{proof}{Proof.}
By definition, \(v ^{\tau +h}(0,x)\) can be written as 
    \[
        v ^{\tau +h}(0,x) = \mathbb{E}\left[\int_0^h e^{-\beta s} f (X_s ^{\tau +h}, u_s^{\tau +h})ds + \int_h^\infty e^{-\beta s} f (X_s ^{\tau +h}, u_s^{\tau +h})ds\Big|\,X^{\tau+h}_0 = x, \mathcal F^B_{\tau+h} \right].
    \]
The second term above
   \begin{align*}
    &\mathbb{E}\left[\int_h^\infty e^{-\beta s} f (X_s ^{\tau +h}, u_s^{\tau +h})ds \Big|\,X^{\tau+h}_0 = x, \mathcal F^B_{\tau+h} \right] \\
    =& \mathbb{E}\left[\mathbb{E}\big[\int_h^\infty e^{-\beta s} f (X_s ^{\tau +h}, u_s^{\tau +h})ds \big| X_h^{\tau+h}, \mathcal{F}^B _{\tau+h}\big]\Big| X_0^{\tau+h}  = x,\mathcal{F}^B _{\tau +h}\right] \nonumber\\
    =& \mathbb{E}\left[e^{-\beta h} \mathbb{E}\big[\int _h^\infty e^{-\beta (s- h)}f(X_s^{\tau+h}, u_s^{\tau +h})ds \big| X_h^{\tau+h},\mathcal{F}_{\tau+h}^B \big]\Big| X_0^{\tau+h} = x, \mathcal{F}^B_{\tau+h}\right]\\ 
    =& \mathbb{E}\left[e^{-\beta h} v^{\tau+h}(h,X_h^{\tau +h})\Big| X_0^{\tau+h} = x,\mathcal{F}_{\tau+h}^B\right]. \nonumber
   \end{align*}
Plugging this back to $v^{\tau+h}(0,x)$ and subtract $v^\tau(0,x)$ on both sides,
   \[
    \begin{aligned}
    &v^{\tau+h}(0,x) - v^\tau (0,x)\\
    &\quad = \mathbb{E}\left[\int_0^h e^{\beta s}f(X_s ^{\tau +h}, u_s^{\tau +h}) ds+ e^{-\beta h} v ^{\tau +h}(h,X_h^{\tau +h}) -v ^ \tau (0,x) \Big| X_0^{\tau +h} =x, \mathcal{F}_{\tau+ h}^B\right].
    \end{aligned}
   \]
   Note that by definition $u^{\tau+h}_{t+h}(x) = u^\tau_t(x) = u_{\tau -t}(x)$ for $t\leq \tau$ and both equal $u_0(x)$ for $t\geq \tau$. Therefore, 
   \begin{equation}
   v^{\tau +h }(h,x) = v ^\tau (0,x ),\label{eq:shift-identity}
   \end{equation}
   and
   \[
    \begin{aligned}
    &v ^{\tau+h}(0,x) - v ^\tau (0,x)\\
    &\quad = \mathbb{E}\left[\int_0^h e^{\beta s}f(X_s ^{\tau +h}, u_s^{\tau +h}) ds+ e^{-\beta h} v ^{\tau}(0,X_h^{\tau +h}) -v ^ \tau (0,x)\Big| X_0^{\tau +h} =x,\mathcal{F}_{\tau+ h}^B\right].
    \end{aligned}
   \]
   Applying it\^o's formula to \(e^{-\beta s} v ^{\tau}(0, X_s^{\tau +h})\) implies,
    \begin{align*}
        d(&e^{-\beta s} v ^{\tau}(0, X_s^{\tau +h})) \nonumber\\
        &=e^{-\beta s}\left(\Big(-\beta v ^{\tau} + b^T(X_s, u^{\tau +h} _s)v^\tau _x + \frac{1}{2}Tr[\sigma \sigma ^T v_{xx}^\tau ]\Big)(0, X_s^{\tau +h})ds +\sigma v_x^\tau(0,X_s^{\tau+h})dW_s\right).
    \end{align*}
    Under Assumption \ref{assumption-T-c}, the stochastic integral is a martingale and vanishes after taking expectation,
    \begin{align*}
        &v^{\tau +h}(0,x) - v ^\tau(0,x)\\
    = &\mathbb{E} \Bigg[ \int_0^h e^{-\beta s} \bigg( -\beta v^\tau(0, X_s^{\tau+h}) + f(X_s^{\tau+h}, u^{\tau+h}_{s}) + b^\top(X_s^{\tau+h}, u^{\tau+h}_{s}) v_x^\tau(0, X_s^{\tau+h}) \\
    & \quad  +\frac{1}{2} \text{Tr}[\sigma \sigma^\top (X_s^{\tau+h}, u_{s}^{\tau+h}) : v_{xx}^\tau(0, X_s^{\tau+h})] \bigg) ds \,\Bigg|\, X_0 = x, \mathcal{F}_{\tau+h}^B \Bigg].
    \end{align*}
    Divide by \(h\) on both sides, let \(h \to 0,\) and apply the dominated convergence theorem under Assumptions \ref{classical-wellposed} and \ref{assumption-T-c},
    \begin{align*}
        d v ^\tau (x) &=\lim\limits_{h \to 0} \frac{v^{\tau +h}(x) - v ^\tau(x)}{h} \nonumber\\
        &= (-\beta v ^\tau (x) + f(x,u_\tau) + b^T(x,u_\tau) v _x^\tau (x) + \frac{1}{2}Tr[\sigma \sigma ^T (x,u_\tau): v _{xx}^\tau(x)]) d\tau \nonumber\\
        & = \left(-\beta v ^\tau (x) + H(x,u_\tau, v  ^\tau_x,v_{xx}^\tau  )\right) d \tau.
    \end{align*}
    The dynamics of $v_x$ and $v_{xx}$ follow from taking derivative on both sides with respect to \(x\). 
\end{proof}

Next we discuss the improvement along the iteration, in terms of the value $v^\tau(x)$.
\begin{lemma}\label{lemma:policy-improve-c}
Suppose for every $\tau>0$, $0\leq s\leq \tau+h$ and sufficiently small $h>0$, $$\mathbb E\left[\|v^{\tau-s+l}_x(X^{\tau+h}_s)\| + \|v^{\tau-s+l}_{xx}(X^{\tau+h}_s)\|+\big|\int_{\mathbb R^n}\|\nabla_u \pi_{\tau-s+l}\|(X^{\tau+h}_s,u)du\big|\right]$$ is bounded in $(\tau-s)^-\leq l \leq h$, then for $\tau_1 < \tau_2$, $\mathbb E[v^{\tau_1}(t,x)] \leq \mathbb E[v^{\tau_2}(t,x)]$, if one of the following conditions holds:
\begin{enumerate}
\item[(i)] $\lambda = 0$, or more generally
\item[(ii)] with $\pi_{\tau-s+l}$ is the probability density function of $u_{\tau-s+l}$ following \eqref{eq:u-dynamics-c}, and $\bH^\tau(x,u) = \bH(x,u,v^\tau_x,v^\tau_{xx})$,
$$\mathbb E\left[\int_0^\tau e^{-\beta s}\left(\|\nabla_u \mathbf{H}^{\tau-s} -\frac{\lambda}{2}\nabla_u \ln\pi_{\tau-s}\|^2 - \frac{\lambda^2}{4}\|\nabla_u\ln\pi_{\tau-s}\|^2 \right)(X^{\tau}_s,u_{\tau-s})ds \right]\geq 0.$$ 
\end{enumerate}
\end{lemma}
\begin{proof}{Proof.}
    We discuss the case of \(t =0, \) while the following argument works for any \(t\geq 0 .\) By It\^o's formula,  
    \begin{align*}
        e^{-\beta t} v^\tau&(t,X_t ^{\tau +h}) - v^\tau (0,x)\\
     = &\int_0^t e^{-\beta s} \left(\Big(v _t^\tau-\beta v ^{\tau} + b^\top(X_s^{\tau +h}, u^{\tau +h} _s)v^\tau _x + \frac{1}{2}Tr[\sigma \sigma ^\top v_{xx}^\tau ]\Big)(0, X_s^{\tau+h})\right)ds \nonumber \\
     &+ \int_0^{t} e^{-\beta s}v _x^\tau (s,X_s^{\tau +h}) \sigma (X_s^{\tau +h}, u_{s}^{\tau+h})dW_s.  
    \end{align*}
    By assumption \ref{assumption-T-c}, taking the expectation leads the integral with respect to \(W\) vanish. Taking expectation on both sides,  
    \begin{align}\label{eq:v-tau-0-x}
        \mathbb{E}[v^\tau(0,x)]&\nonumber = \mathbb{E}\left[\int_0^{\tau+h} e^{-\beta s}\Big(-v_t^\tau+\beta v ^{\tau} - b^\top(X_s^{\tau +h}, u^{\tau +h}_s)v^\tau _x\Big)(0,X_s^{\tau+h}) ds\right] \nonumber\\
        -&\frac{1}{2} \mathbb{E}\Big[\int_0^{\tau+h}e^{-\beta s}\frac{1}{2}Tr[\sigma \sigma^\top v_{xx}^\tau(0, X_s^{\tau+h}) ]ds\Big]  + \mathbb{E}\left[e^{-\beta(\tau +h)}v^\tau (\tau + h, X_{\tau +h}^{\tau +h})\right]. 
    \end{align} 
    Note that \(M_t := \int_0^t e^{-\beta s} f(X_s^\tau, u_s^\tau)ds + e^{-\beta t} v^\tau(t,X_t^\tau)= \mathbb E[\int_0^\infty f(X_s^\tau, u_s^\tau)ds|\mathcal F^{W}_t,\mathcal F^B_\tau]\) is a martingale with respect to $\{\mathcal F^W_t\}_{t\geq 0}$, and thus
    \[
        v _t^\tau (t,x) - \beta v ^\tau (t,x) + b ^T (x, u_t^\tau) v _x^\tau (t,x) + \frac{1}{2} Tr[\sigma \sigma ^T (x, u_t^\tau) : v_{xx}^\tau (t,x)] + f(x,u_t^\tau) = 0.
    \]  
    Substituting $v^\tau_t$ into \eqref{eq:v-tau-0-x}, 
    \begin{align}
        &\mathbb{E}[v ^\tau(0,x)]\nonumber\\
    =&\ \mathbb{E} \left[ \int_{0}^{\tau+h} e^{-\beta s} \left(b(X_{s}^{\tau+h}, u_{\tau-s}) - b(X_{s}^{\tau+h}, u_{\tau+h-s})^{\top} v_{x}^{\tau}(s, X_{s}^{\tau+h})\right) ds \right]\nonumber \\
    &+ \frac{1}{2} \mathbb{E} \left[ \int_{0}^{\tau+h} e^{-\beta s} \text{Tr}\left[ \left(\sigma\sigma^{\top}(X_{s}^{\tau+h}, u_{\tau-s}) - \sigma\sigma^{\top}(X_{s}^{\tau+h}, u_{\tau+h-s})\right) : v_{xx}^{\tau}(s, X_{s}^{\tau+h}) \right] ds \right] \nonumber\\
    &+ \mathbb{E} \left[ \int_{0}^{\tau+h} e^{-\beta s} f(X_{s}^{\tau+h}, u_{\tau-s}) ds \right] + \mathbb{E} \left[ e^{-\beta(\tau+h)} v^{\tau}(\tau+h, X_{\tau+h}^{\tau+h}) \right]\nonumber\\
    =&\ \mathbb{E}[v^{\tau+h}(0, x)] + \mathbb{E} \left[ \int_{0}^{\tau+h} e^{-\beta s} \left(f(X_{s}^{\tau+h}, u_{\tau-s}) - f(X_{s}^{\tau+h}, u_{\tau+h-s})\right) ds \right] \nonumber\\
    &+ \mathbb{E} \left[ \int_{0}^{\tau+h} e^{-\beta s} \left(b(X_{s}^{\tau+h}, u_{\tau-s}) - b(X_{s}^{\tau+h}, u_{\tau+h-s})^{\top} v_{x}^{\tau-s}(X_{s}^{\tau+h})\right) ds \right] \nonumber \\
    &+ \frac{1}{2} \mathbb{E} \left[ \int_{0}^{\tau+h} e^{-\beta s} \text{Tr}\left[ \left(\sigma\sigma^{\top}(X_{s}^{\tau+h}, u_{\tau-s}) - \sigma\sigma^{\top}(X_{s}^{\tau+h}, u_{\tau+h-s})\right) : v_{xx}^{\tau-s}(X_{s}^{\tau+h}) \right] ds \right], \label{v-tau-0-rearrange}
    \end{align} 
    where we use the fact that \(v ^\tau (s, x) =v^{\tau-s}(0,x)\) from \eqref{eq:shift-identity}. Notice that  all three terms above involves change in functions for different controls. Apply It\^o formula with respect to the dynamics of $u_\tau$, 
    \begin{align*}
        &\mathbb{E}\left[f(X^{\tau +h}_s,u_{\tau-s}) - f(X^{\tau +h}_s, u_{\tau +h -s})\right] \nonumber\\
        &= - \mathbb{E}\left[\int _0^h ((\nabla_u f)^\top\nabla_u H^{\tau -s +l} + \lambda Tr[\nabla_{uu}^2 f]) (X_s^{\tau +h}, u_{\tau -s +l}) dl\right];\\
        &\mathbb{E}\left[(b(X_s, u_{\tau -s}) -b(X_s, u_{\tau +h -s}))^\top v_x^{\tau-s}(X_s^{\tau +h})\right]\nonumber\\
        &= \mathbb{E}\left[-\int_0^h \Big((\nabla_{u} b \nabla_u H^{\tau -s+l})^\top v_x^{\tau-s} + \lambda \sum\limits_{i=1}^{d} Tr[\nabla^2_{uu}b_i]v^{\tau-s}_{x_i}\Big)(X_s^{\tau +h}, u_{\tau-s+l})dl\right]; \\
        &\mathbb{E}\left[Tr[(\sigma \sigma^\top (X_{s}^{\tau+h},u_{\tau-s}) -\sigma \sigma ^\top(X_s^{\tau+h}, u_{\tau+h-s})):v_{xx}^{\tau-s} (X_s ^{\tau+h}, u_{\tau -s+l})]dl\right]\nonumber\\
        &= \mathbb{E}\left[ - \int_0^h(\nabla_u Tr[\sigma \sigma ^T:v_{xx}^{\tau-s} ]^\top \nabla_u H^{\tau -s +l} +\lambda  \nabla^2_{uu}Tr[\sigma \sigma ^T: v_{xx}^{\tau-s}])(X_s^{\tau+h}, u_{\tau-s+l})dl\right].
    \end{align*}
    Plugging this back to \eqref{v-tau-0-rearrange},
    \begin{align}
        \mathbb{E}[v^\tau (0,x)& - v ^{\tau+h}(0,x)]\nonumber\\
        =& -\mathbb{E}\left[\int_0^{\tau +h}e^{-\beta s}(\int_0^h (\nabla_u H^{\tau-s+l})^\top \nabla_u H^{\tau-s} (X_s^{\tau +h}, u_{\tau -s +l})dl )ds\right]\nonumber\\
        &-\lambda \mathbb{E}\left[\int_{0}^{\tau+h} e^{-\beta s} \int _0^h  Tr[\nabla^2_{uu} H^{\tau-s}](X_s^{\tau+h},u_{\tau-s+l})dl ds\right] \label{eq:tr-deltauu-H},
    \end{align} 
    with convention that all the terms evaluated at \(\tau - s + l \leq  0\) is its initial value at $0$. For the second term, integration by parts implies
    \[
        \mathbb{E}\left[Tr[\nabla^2_{uu} H^{\tau-s}](X_s^{\tau+h},u_{\tau-s+l})| \mathcal F^W_s\right] = - \mathbb{E}\left[(\nabla H ^{\tau-s})^\top \nabla_u \ln \pi_{\tau -s +l}(X_s^{\tau +h}, u_{\tau- s+l})| \mathcal F^W_s\right].
    \]
    where \(\pi _{\tau -s + l}\) is the distribution of \(u_{\tau-s+l}.\) Plugging this into \eqref{eq:tr-deltauu-H} and dividing by \(h\) on both sides implies:
    \begin{align*}
        & \frac{\mathbb{E}[v ^{\tau +h}(0,x) - v^\tau (0,x)]}{h}\nonumber\\
        =&\mathbb{E}\left[\int_0^{\tau +h}e^{-\beta s} \Big(\int _0^h \frac{1}{h} \int_0^h (\nabla_u H^{\tau-s+l})^\top \nabla_u H^{\tau-s}(X_s^{\tau +h}, u_{\tau -s +l})dl\Big) ds\right] \nonumber\\
        & -\lambda \mathbb{E}\left[ \int_0^{\tau +h}e ^{-\beta s} \int _0^h \frac{1}{h} (\nabla_u H^{\tau-s})^\top \nabla_u \ln \pi _{\tau -s +l}(X_s^{\tau+h}, u_{\tau -s+l})dlds\right]
        \end{align*}
    Let $h\rightarrow 0$ and by the dominated convergence theorem, 
     \begin{align*}
        &\frac{d \mathbb{E}[v ^\tau (0,x)]}{d \tau} \\
        = &\mathbb{E}\left[\int_0^{\tau}e^{-\beta s} \left\lVert\nabla_u H^{\tau-s})\right\rVert^2(X_s^{\tau}, u_{\tau -s}) ds\right]  -\lambda \mathbb{E}\left[ \int_0^{\tau +h}e ^{-\beta s} (\nabla_u H^{\tau-s})^\top \nabla_u \ln \pi_{\tau -s}(X_s^{\tau}, u_{\tau -s})ds\right]\\
        =&\mathbb{E}\Big[\int_0^\tau e^{-\beta s}\left\lVert \nabla_u H^{\tau -s} - \frac{\lambda }{2} \nabla_u \ln \pi _{\tau -s}\right\rVert^2 (X_s^\tau, u_{\tau-s}) ds \Big] -\mathbb{E}\Big[\int_0^\tau e^{-\beta s} \frac{\lambda ^2}{4} \left\lVert \nabla_u \ln \pi _{\tau-s}\right\rVert^2 (X_s^{\tau}, u_{\tau-s})ds \Big].
    \end{align*}
     If \(\lambda = 0\), then \( \frac{d \mathbb{E}[v ^\tau (0,x)]}{d \tau} \geq 0\) and therefore \(\mathbb{E}[v ^\tau (x)]\) is non-decreasing in \(\tau.\) Otherwise, the condition in (ii) also guarantees the improvement over $\tau$.
\end{proof}
In the rest of this section we focus on the case of $\lambda = 0$, where the iterations of the value function deterministic. In this case, the previous lemma shows that we have guaranteed policy improvement and thus the value iteration converges. The next proposition provides a sufficient condition for the optimality in the limit. The case of $\lambda \neq 0$ will be discussed in the next section after we introduce the synchronous coupling in the next section. The convergence of $u_\tau$,$v^\tau_x$ and $v^\tau_{xx}$, can be guaranteed by the Lipschitz conditions in expectation we introduce in the next section.
\begin{proposition}\label{prop:limit-det}
	With $\lambda = 0$, (i) $v^\tau$ converges to $v^*$, and (ii) If $v_x^\tau$ and $v^\tau_{xx}$ also converge, and $$\nabla_u (\mathbf{H}(x,u,v^*_x(x),v^*_{xx}(x)))|_{|u|=\infty} \neq 0,$$ then $u_\tau$ (up to a subsequence) converges to $u^*$ and $\nabla_u (\mathbf{H}(x,u^*,v^*_x(x),v^*_{xx}(x))) =0$. If in addition,
	$$u^*(x) = \argmax\limits_u \mathbf{H}(x,u,v^*_x(x),v^*_{xx}(x)),$$
	and $\hat u = \{u_t = u^*, t\geq 0\} \in \mathcal A(x)$, then $\hat u$ is the optimal strategy for the stochastic control problem $V(x) =  \sup\limits_{u\in\mathcal A(x)} J(x,u)$.
\end{proposition}
\begin{proof}{Proof.}
	(i) If $\lambda = 0$, then $u_\tau$ and thus $v^\tau$ are deterministic. Lemma \ref{lemma:policy-improve-c} implies that $v^\tau$ is non-decreasing in $\tau$, and because the value function is finite, it converges to $v^*$. 
    
    (ii) For each $\tau$, $v^\tau(t,x)$ solves
    \begin{align*}
        v_t^\tau - \beta v^\tau + b(x, u_t^\tau (x))^\top v_x^\tau + \frac{1}{2} \sigma(x, u_t^\tau(x)) \sigma(x, u_t^\tau(x))^\top :v_{xx}^\tau + f (x, u_t^\tau (x) )  =0 , \quad 0 < t < \tau. 
    \end{align*}
    By definition of \(v^\tau\), for every fixed \(t \ge 0,\)
    \begin{align*}
        v^*(t,x) = \lim_{\tau \to \infty} v^{\tau -t} (0,x) =v^*(0,x),
    \end{align*}
    therefore the limit is independent of t, and \(v^*_t = 0.\) Furthermore, since the monotone function $v^\tau(x)$ converges, \eqref{eq:u-dynamics-c} implies that $$ -\beta v^\tau(x) + f(x,u_\tau) + b^\top(x,u_\tau)v^\tau_x(x) + \frac{1}{2}\tr[\sigma\sigma^\top(x,u_\tau):v^\tau_{xx}(x)],$$ converges to $0$. In other words, for every $x$
	\begin{equation*}
		g_\tau = f(x,u_\tau) + b^\top(x,u_\tau)v^\tau_x(x)+\frac{1}{2}\tr[\sigma\sigma^\top(x,u_\tau):v^\tau_{xx}(x)],
	\end{equation*}
	converges to constant $\beta v^*(x)$.
	\begin{equation*}
		dg_\tau =\|\nabla_u \mathbf{H}^\tau(x,u_\tau)\|^2d\tau +  b^\top(x,u_\tau)dv^\tau_x(x) +\frac{1}{2}\tr[\sigma\sigma^\top(x,u_\tau):dv^\tau_{xx}(x)].
	\end{equation*}
	Since $g_\tau, v^\tau_x(x)$ and $v^\tau_{xx}(x)$ converge in $\tau$, their derivative with respect to $\tau$ should converge to $0$,
	which implies that $\nabla_u \mathbf{H}^\tau(x,u_\tau)$ converges to $0$. Since $\nabla_u(\mathbf{H}(x,u,v^*_x(x),v^*_{xx}(x)))|_{|u|=\infty} \neq 0$, $\{u_\tau\}_{\tau\geq 0}$ is bounded, and thus up to a subsequence converges to $u^*$, and by the continuity of $\mathbf{H}$, $\nabla_u (\mathbf{H}(x,u^*,v^*_x(x),v^*_{xx}(x))) =0$.

	Finally, along the converging subsequence, from \eqref{eq:v-dynamics-c},
	\begin{align*}
		0 =& \lim_{\tau\rightarrow\infty}\left(-\beta v^\tau(x) + f(x,u_\tau) + b^\top(x,u_\tau) v^\tau_x(x)+ \frac{1}{2}\tr[\sigma\sigma^\top(x,u_\tau):v^\tau_{xx}(x)]\right)\nonumber\\
		=& -\beta v^*(x) + f(x,u^*) + b^\top(x,u^*)v^*_x(x) + \frac{1}{2}\tr[\sigma\sigma^\top(x,u^*):v^*_{xx}(x)].
	\end{align*}
If $u^*(x) = \argmax\limits_u \mathbf{H}(x,u(x),v^*_x(x),v^*_{xx}(x))$, then $\sup_{u}\left(-\beta v^*(x) + \mathbf{H}(x,u,v^*_x,v^*_{xx})\right) = 0$. Thus $v^*$ satisfies the HJB equation corresponding to the control problem $V(x) = \sup\limits_{u\in\mathcal A(x)} J(x,u)$, with $u^*$ being the maximizer of the associated Hamiltonian. The verification of optimality follows standard arguments and is omitted here.
\end{proof}
The previous Proposition shows that if $\mathbf{H}$ is concave in $u$, then $u^*$ is the optimal control, e.g. if $f$ is concave and $b$ and $\sigma$ are linear in $u$, which implies that $v(x,u)$ for any $u\in\mathcal A(x)$ is concave in $x$, as discussed in Section \ref{sec:power-utility-eg}. In the case of non-concave Hamiltonian $\mathbf{H}$, with $\lambda = 0$, it is possible that $u^*$ is stuck in a local maximum of $\bH(x,u,v^*_x,v^*_{xx})$. Assuming that the policy iteration has a limit, we propose the following criterion to check for optimality, which involves a sequence of policy and value iterations. This criterion applies to any policy iteration algorithm, beyond \eqref{eq:u-dynamics-c}, for the inhomogeneous Hamiltonian coupled with value iteration \eqref{eq:v-dynamics-c}.
\begin{definition}\label{def:sequential}
 Given initial values $u_0$ and $v_0(x) = J(x,u^0)$, where $u^0 = \{u_t = u_0, t\geq 0\}$, denote as $u^{(1)}_\tau$ and $v^{(1)\tau}$ the policy and value following \eqref{eq:u-dynamics-c} $($or any other policy iteration algorithm$)$ and \eqref{eq:v-dynamics-c}, and denote as $u^{(1)}_*$ and $v^{(1)*}$ the limit of these iterations. For $i \geq 1$, if $u^{(i)}_*(x) \neq \argmax\limits_u \mathbf{H}(x,u,v^{(i)*}_x(x),v^{(i)*}_{xx}(x))$, then let $$u^{(i+1)}_0(x) = \argmax\limits_u \mathbf{H}(x,u,v^{(i)*}_x(x),v^{(i)*}_{xx}(x)),$$ and otherwise let $u^{(i+1)}_0(x) = u^{(i)}_*$. Denote as $u^{(i+1)}_\tau$ and $v^{(i+1)\tau}$ the policy and value following \eqref{eq:u-dynamics-c} $($or any other policy iteration algorithm$)$ and \eqref{eq:v-dynamics-c} with initial values $u^{(i+1)}_0(x)$ and $v^{(i)*}$, and the limits are denoted as $u^{(i+1)}_*$ and $v^{(i+1)*}$.
\end{definition}
According to the above definition, if the limit of the policy iteration $u^{(i)}_*$ is the maximizer of the Hamiltonian corresponding to the limiting value $v^{(i)*}$ for every $x$ (so that it is the optimal control), then the iteration simply stops. If not, then the optimality fails, because we can choose an initial value $u^{(i+1)}_0$ which gives a higher value of the same Hamiltonian, and restart the iteration. The next proposition shows the policy improvement holds in the above sequential iterations.
\begin{proposition}\label{prop:control-update-classic}
For any $i\geq 1$, if $u^{(i+1)}_0 \neq u^{(i)}_*$, then $dv^{(i+1)\tau}(x)|_{\tau = 0} >0$. In particular, if the policy iteration follows \eqref{eq:u-dynamics-c} with $\lambda= 0$, then $v^{(i+1)\tau_j} > v^{(i)\tau_i}$, where $\tau_i,\tau_j\in[0,\infty]$.
\end{proposition}
\begin{proof}{Proof.}
Since $v^{(i+1)0} = v^{(i)*}$, 
\begin{align*}
	&dv^{(i+1)\tau}(x)|_{\tau = 0} \nonumber\\
    =& \left(-\beta v^{(i)*}(x) + f(x,u^{(i+1)}_0) + b^\top(x,u^{(i+1)}_0) v^{(i)*}_x(x) + \frac{1}{2}\tr[\sigma\sigma^\top(x,u^{(i+1)}_0):v^{(i)*}_{xx}(x)]\right)d\tau.
\end{align*}
By Definition \ref{def:sequential},
$$\mathbf{H}(x,u^{(i+1)}_0(x),v^{(i)*}_x(x),v^{(i)*}_{xx}(x)) > \mathbf{H}(x,u^{(i)*}(x),v^{(i)*}_x(x),v^{(i)*}_{xx}(x)).$$
Thus,
\begin{equation*}
    \begin{aligned}
    &dv^{(i+1)\tau}(x)|_{\tau = 0}\\
    &\quad > \left(-\beta v^{(i)*}(x) + f(x,u^{(i)}_*) + b^\top(x,u^{(i)}_*)v^{(i)*}_x(x) + \frac{1}{2}\tr[\sigma\sigma^\top(x,u^{(i)}_*):v^{(i)*}_{xx}]\right)d\tau.
    \end{aligned}
\end{equation*}
On the other hand, since $(u^{(i)}_*, v^{(i)*})$ is the limit of the $i$-th round of iteration,
$$-\beta v^{(i)*}(x) + f(x,u^{(i)}_*) + b^\top(x,u^{(i)}_*)v^{(i)*}_x(x) + \frac{1}{2}\tr[\sigma\sigma^\top(x,u^{(i)}_*):v^{(i)*}_{xx}] = 0,$$ which implies that $dv^{(i+1)\tau}(x)|_{\tau = 0} >0$.
If the policy iteration follows \eqref{eq:u-dynamics-c} with $\lambda=0$, Lemma \ref{lemma:policy-improve-c} implies that $v^{(i)\tau_1} \leq v^{(i)\tau_2}$ for any $i\geq 1$ and $\tau_1 < \tau_2 \leq \infty$. Thus in order to prove the claim, it suffices to check $dv^{(i+1)\tau}>0$ at $\tau = 0$ as above, because that is the only instance where we manually choose the drift, in particular the control $u$ in the Hamiltonian $\mathbf{H}$, of the value iteration, which otherwise should be the current iterates for $u$, according to \eqref{eq:u-dynamics-c}.
\end{proof}

\section{Convergence analysis}\label{sec: convergence}
In this section, we discuss the convergence rate of the continuous policy \eqref{eq:u-dynamics-c} and value \eqref{eq:v-dynamics-c} dynamics proposed in the previous section with temperature $\lambda >0$, under coupling bound assumptions. Because \(u_\tau\) is iteslef a function of \(x\), we define the following quantities for computational purpose,
    \begin{align*}
        p_\tau := d_x u_\tau(x) \in \mathbb{R}^{n \times d}, \quad q_\tau := d^2_{xx} u_\tau(x). 
    \end{align*}
    Then rewrite the iterations as following 
    \begin{align}
        d_\tau v_x^\tau &= -\beta v_x^\tau + [d_x \mathbf{H}]_u + p_\tau^\top \nabla_u \mathbf{H}\label{eq:v_x-with-pq}\\
        d_\tau v_{xx}^\tau &= -\beta v^\tau_{xx} + [d^2_{xx}\mathbf{H}]_u +2 sym\left [p_\tau^\top [d_x \nabla_u \mathbf{H}]_u\right] + p^\top_\tau (\nabla_{uu}^2 \mathbf{H}) p_\tau + q^\top_\tau \nabla_u \mathbf{H}.\label{eq:v_xx-with-pq}
    \end{align}
where \([\cdot]_u\) is the operation by fixing policy \(u.\) and $p_\tau$ and $q_\tau$ should have dynamics on their own. Differentiating \eqref{eq:u-dynamics-c} in $x$ gives
    \begin{align}
    \partial_\tau p_\tau&=\big(\nabla^2_{uu}\mathbf{H}\big)p_\tau
    +\big[\partial_x\nabla_u\mathbf{H}\big]_{u},\label{eq:p-dyn}\\
    \partial_\tau q_\tau&=\big(\nabla^2_{uu}\mathbf{H}\big)q_\tau
    +\big[\partial^2_{xx}\nabla_u\mathbf{H}\big]_{u}
    +2sym\!\Big[\big[\partial_x\nabla^2_{uu}\mathbf{H}\big]_{u}p_\tau\Big]
    +p_\tau^\top\big(\nabla^3_{uuu}\mathbf{H}\big)p_\tau.\label{eq:q-dyn}
    \end{align}
$\bW^\tau:=\big(v^\tau(x),\,v^\tau_x(x),\,v^\tau_{xx}(x),\,p_\tau,\,q_\tau\big) \in \mathbb{R}^N$ where \(N = 1+ d + d^2 + nd +nd^2\), and the norm $\|\bW\|^2:=|v|^2+\|v_x\|^2+\|v_{xx}\|_F^2+\|p\|^2+\|q\|^2$,
we group these iteration dynamics as
\begin{equation}\label{eq:Z}
\begin{cases}
\;du_\tau=\nabla_u\bH(x,u_\tau,\bW^\tau)\,d\tau+\sqrt{2\lambda}\,dB_\tau,\\[2pt]
\;d\bW^\tau=\big(-\beta \bW^\tau+\cG(x,u_\tau,\bW^\tau)\big)\,d\tau ,
\end{cases}
\end{equation}
where
\begin{align*}
\mathcal{G}(x, u, \bW) = \begin{bmatrix} 
H(x, u, v_x, v_{xx}) \\ 
[\partial_x H]_u + p^\top \nabla_u H \\ 
[\partial_{xx}^2 H]_u + 2\text{sym}[p^\top [\partial_x \nabla_u H]_u] + p^\top (\nabla_{uu}^2 H) p + q^\top \nabla_u H \\ 
(\beta I + \nabla_{uu}^2 H)p + [\partial_x \nabla_u H]_u \\ 
(\beta I + \nabla_{uu}^2 H)q + [\partial_{xx}^2 \nabla_u H]_u + 2\text{sym}[[\partial_x \nabla_{uu}^2 H]_u p] + p^\top (\nabla_{uuu}^3 H) p 
\end{bmatrix}.
\end{align*}
Notice that the system of \(\bW^\tau\) may not be a closed, because differentiating the value equation once and twice in space introduces \(v_{xxx}\) and \(v_{xxxx}\) respectively. In the following we show that for sufficiently smooth policy and value along the coupled iteration, \(v_{xxx}\) and \(v_{xxxx}\) are uniformly bounded in \(L^2\) and prevent the higher derivatives from growing without control as the iteration proceeds.

To establish convergence of the coupled policy–value iteration, the following quantitative coupling conditions are imposed. These assumptions provide (one-sided) Lipschitz-type bounds that control the dependence of the policy and value dynamics on each other. Notice that these conditions and the discussions below do not depend on the temperature $\lambda$ being a constant, and applies with annealed $\lambda$ discussed later.
\begin{assumption}\label{assum:coupling-bounds}
For every $x\in\R^d$, there exist $\kappa(x)>0$,
$L_{VV}(x)\in\R$, $L_{uV}(x)>0$, and $L_{Vu}(x)\ge0$ such that for all $u,\tilde u\in\R^n$ and
$\bW,\tilde\bW$,
\begin{align}
(\mathrm{H1})\quad& \mathbb{E}\big\langle\cG(x,u,\bW)-\cG(x,u,\tilde \bW),\,\bW-\tilde \bW\big\rangle
\le L_{VV}(x) \mathbb{E}\|\bW-\tilde \bW\|^2;\tag{H1}\label{eq:H1}\\
(\mathrm{H2})\quad&\mathbb{E}\big\|\cG(x,u,\bW)-\cG(x,\tilde u,\bW)\big\|^2\le L_{uV}(x)^2 \mathbb{E}\|u-\tilde u\|^2;\tag{H2}
\label{eq:H2}\\
(\mathrm{H3})\quad& \mathbb{E}\left\langle\nabla_u\bH(x,u,\bW)-\nabla_u\bH(x,\tilde u,\bW),\,u-\tilde u\right\rangle
\le-\kappa(x) \mathbb{E}\|u-\tilde u\|^2;\tag{H3}\label{eq:H3}\\
(\mathrm{H4})\quad&\mathbb{E}\big\|\nabla_u\bH(x,u,\bW)-\nabla_u\bH(x,u,\tilde \bW)\big\|^2
\le L_{Vu}(x)^2\mathbb{E}\|\bW-\tilde \bW\|^2.\tag{H4}\label{eq:H4}
\end{align}
\end{assumption}

For \(c > 0,\) define the weighted norm on \(\mathbb{R}^n \times \mathbb{R}^N\) by \(\|(u,\mathbf W)\|_c^2:=\|u\|^2+c\|\mathbf W\|^2.\)
Denote by
$\mathcal P_{2}(\mathbb R^{n} \times \mathbb{R}^N)$ the space of Borel probability measures with
finite second moment, equipped with the $2$-Wasserstein distance
\begin{align}
\mathcal W_{2,c}(\mu,\nu)
=\Big(\inf_{\gamma\in\Pi(\mu,\nu)}
\int_{  \mathbb R^{n}\times\mathbb R^{N}}\left (\|u-\tilde{u}\|^{2}+c\|\mathbf W-\widetilde{\mathbf W}\|^2
\right )d\gamma \Big)^{1/2},
\end{align}
where $\Pi(\mu,\nu)$ denotes the set of couplings of $\mu,\nu$ (c.f.
\citet{Villani2009,AmbrosioGigliSavare2008}). Under the above assumptions, we first show the exponential convergence of the continuous dynamics \eqref{eq:Z},
\begin{theorem}\label{thm:joint-contraction}
Fix $x\in\R^d$ and let Assumption \ref{assum:coupling-bounds} hold. 
Let $Z_\tau=(u_\tau,\bW^\tau)$ and $\tilde Z_\tau=(\tilde u_\tau,\tilde \bW^\tau)$ solve
\eqref{eq:Z} driven by the same Brownian motion $B$, from possibly random initial
data $Z_0,\tilde Z_0$. Define 
\begin{equation}\label{eq:alpha}
\Phi_\tau:=\|u_\tau-\tilde u_\tau\|^2+\alpha^*\|\bW^\tau-\tilde \bW^\tau\|^2,
\end{equation}
where $\alpha^*:=\frac{L_{Vu}(x)}{L_{uV}(x)}$.
If the following assumption hold:
\begin{align}
 &\kappa(x)\big(\beta-L_{VV}(x)\big)>L_{uV}(x)\,L_{Vu}(x);\label{eq:smallgain}\\
 &\sup_{0 \le \tau \le T} \mathbb{E}\left[  \left\lVert u_\tau\right\rVert^2 + \left\lVert \bW^\tau\right\rVert^2 + \left\lVert \nabla_u H(x,u_\tau, \bW^\tau)\right\rVert^2 + \left\lVert -\beta \bW^\tau + \cG(x, u_\tau, \bW^\tau)\right\rVert ^2   \right] < \infty. \label{ass:bouned-all}
\end{align}
Then, almost surely and for every $\tau\ge0$,
\begin{equation}\label{eq:contraction}
\Phi_\tau\le e^{-2\rho^*(x)\tau}\Phi_0,\text{ with }
\rho^*:=\frac{\kappa+\beta - L_{VV}}{2}-\sqrt{\Big(\frac{\kappa-\beta + L_{VV}}{2}\Big)^{2}+L_{uV}L_{Vu}}\;>\;0 .
\end{equation}
Consequently, if $\cL_\tau,\tilde\cL_\tau$ denote the laws of $Z_\tau,\tilde Z_\tau$ on
$\R^n\times \R^N$, the Wasserstein distance satisfies
\begin{equation}\label{eq:W2}
\cW_{2,\alpha^*}\big(\cL_\tau,\tilde\cL_\tau\big)
\le e^{-\rho^*(x)\tau}\,\cW_{2,\alpha^*}\big(\cL_0,\tilde\cL_0\big).
\end{equation}
    
\end{theorem}

\begin{proof}{Proof.}
Denote as $m_1:=\kappa(x)$, $m_2:=\beta-L_{VV}(x)$ and $g:=L_{uV}(x)L_{Vu}(x)$.
    The couplings are driving by the same Brownian motion \(B\) and the diffusion coefficient of \(u\) is constant, therefore the difference solves an ordinary differential equation (ODE) with random coefficients
    \begin{align*}
        \frac{d}{d \tau} (u_\tau- \tilde{u}_\tau) = \nabla_u H(x, u_\tau, \bW^\tau)- \nabla_u H(x, \tilde{u}_\tau , \tilde{\bW}^\tau).
    \end{align*}
    By using \eqref{eq:H3} and \eqref{eq:H4} and Cauchy-Schwartz inequality, the coupling difference of policy iteration can be bounded as following,
    \begin{align*}
        \frac{d}{d\tau} \mathbb{E}\left\lVert u_\tau - \tilde{u}_\tau \right\rVert^2  &= 2 \mathbb{E}\left\langle u_\tau - \tilde{u}_\tau , \frac{d}{d \tau} (u_\tau - \tilde{u}_\tau) \right\rangle\\
        & = 2 \mathbb{E}\left\langle u_\tau - \tilde{u}_\tau , \nabla_u H(x, u_\tau, \bW^\tau)- \nabla_u H(x, \tilde{u}_\tau , \tilde{\bW}^\tau) \right\rangle\\
        & = 2 \mathbb{E}\left\langle u_\tau - \tilde{u}_\tau , \nabla_u H(x, u_\tau, \bW^\tau)- \nabla_u H(x, \tilde{u}_\tau , \bW^\tau) \right\rangle \\
        &+ 2 \mathbb{E}\left\langle u_\tau - \tilde{u}_\tau , \nabla_u H(x, \tilde{u}_\tau, \bW^\tau)- \nabla_u H(x, \tilde{u}_\tau , \tilde{\bW}^\tau) \right\rangle\\
        &\le -2 m_1 \mathbb{E}\left\lVert u_\tau - \tilde{u}_\tau \right\rVert^2 + 2\mathbb{E}  \left[\left\lVert u_\tau - \tilde{u}_\tau \right\rVert \left\lVert \nabla_u H(x, u_\tau, \bW^\tau)- \nabla_u H(x, \tilde{u}_\tau , \bW^\tau) \right\rVert \right]\\
        &\le -2 m_1 \mathbb{E}\left\lVert u_\tau - \tilde{u}_\tau \right\rVert^2 \\& \quad+ 2(\mathbb{E}  \left\lVert u_\tau - \tilde{u}_\tau \right\rVert ^2)^{\frac{1}{2}} (\mathbb{E}\left\lVert \nabla_u H(x, u_\tau, \bW^\tau)- \nabla_u H(x, \tilde{u}_\tau , \bW^\tau) \right\rVert^2)^{\frac{1}{2}} \\\
        &\leq -2 m_1 \mathbb{E}\left\lVert u_\tau - \tilde{u}_\tau \right\rVert^2 + 2 L_{Vu} \left(\mathbb{E}  \left\lVert u_\tau - \tilde{u}_\tau \right\rVert ^2\right)^{\frac{1}{2}} \left(\mathbb{E} \left\lVert \bW^\tau - \tilde{\bW}^\tau \right\rVert ^2\right)^{\frac{1}{2}}.
    \end{align*}
    Similarly, by using \eqref{eq:H1} and \eqref{eq:H2} and the dynamic of \(\bW\) from \eqref{eq:Z},
    \begin{align*}
        &\frac{d}{d\tau} \mathbb{E}\left\lVert \bW ^\tau- \tilde{\bW}^\tau \right\rVert^2 =  2 \mathbb{E}\left\langle\bW^\tau - \tilde{\bW}^\tau, -\beta(\bW^\tau - \tilde{\bW}^\tau) + \cG(u_\tau,\bW^\tau) - \cG (\tilde{u}_\tau, \tilde{\bW}^\tau)\right\rangle\\
        &= -2 \beta  \mathbb{E}\left\lVert \bW^\tau - \tilde{\bW}^\tau \right\rVert^2 + 2 \mathbb{E}\left\langle \bW^\tau -\tilde{\bW}^\tau, \cG(u_\tau, \bW^\tau) - \cG(u_\tau, \tilde{\bW}^\tau)\right\rangle \\
        & +2 \mathbb{E}\left\langle \bW^\tau -\tilde{\bW}^\tau, \cG(u_\tau, \tilde{\bW}^\tau) - \cG(\tilde{u}_\tau, \tilde{\bW}^\tau)\right\rangle\\
        &\leq -2 \beta  \mathbb{E}\left\lVert \bW^\tau - \tilde{\bW}^\tau \right\rVert^2 +2L_{VV} \mathbb{E}\left\lVert \bW^\tau - \tilde{\bW}^\tau \right\rVert^2 + 2L_{uV} \left\lVert u_\tau - \tilde{u}_\tau\right\rVert\mathbb{E}\left\lVert \bW^\tau - \tilde{\bW}^\tau \right\rVert\\
        &\leq -2 \beta  \mathbb{E}\left\lVert \bW^\tau - \tilde{\bW}^\tau \right\rVert^2 +2L_{VV} \mathbb{E}\left\lVert \bW^\tau - \tilde{\bW}^\tau \right\rVert^2  + 2L_{uV} \left(\mathbb{E}\left\lVert u_\tau - \tilde{u}_\tau\right\rVert^2\right)^{\frac{1}{2}}\left(\mathbb{E}\left\lVert \bW^\tau - \tilde{\bW}^\tau \right\rVert^2\right)^{\frac{1}{2}} \\
        &= -2m_2 \mathbb{E}\left\lVert \bW^\tau - \tilde{\bW}^\tau \right\rVert^2 + 2L_{uV} \left(\mathbb{E}\left\lVert u_\tau - \tilde{u}_\tau\right\rVert^2\right)^{\frac{1}{2}}\left(\mathbb{E}\left\lVert \bW^\tau - \tilde{\bW}^\tau \right\rVert^2 \right)^\frac{1}{2}.
    \end{align*}
    For arbitrary \(\alpha >0\), set \(\Phi := \left\lVert u - \tilde{u} \right\rVert^2 + \alpha \left\lVert \bW - \tilde{\bW} \right\rVert^2\) and \(C_\alpha = L_{Vu} + \alpha L_{uV}.\) Then above inequalities implies that for any $\rho^* >0$,
    \begin{align*}
        &\frac{d}{d \tau} \mathbb{E}\Phi_\tau \leq -2m_1 \mathbb{E}\left\lVert u_\tau - \tilde{u}_\tau \right\rVert^2 - 2\alpha m_2 \mathbb{E}\left\lVert \bW ^\tau- \tilde{\bW}^\tau \right\rVert^2 + 2 C_\alpha \left(\mathbb{E}\left\lVert u_\tau - \tilde{u}_\tau \right\rVert^2\right)^{\frac{1}{2}} \left( \mathbb{E}\left\lVert \bW ^\tau- \tilde{\bW}^\tau \right\rVert \right)^{\frac{1}{2}}\\
        &=-2 \rho^* \mathbb{E}\Phi_\tau \\
        &- 2\left((m_1 - \rho^*) \mathbb{E}\left\lVert u_\tau - \tilde{u}_\tau \right\rVert^2 - C_\alpha \left(\mathbb{E}\left\lVert u_\tau - \tilde{u}_\tau \right\rVert^2\right)^{\frac{1}{2}} \left( \mathbb{E}\left\lVert \bW ^\tau- \tilde{\bW}^\tau \right\rVert \right)^{\frac{1}{2}}+ \alpha(m_2 - \rho^*) \left\lVert \bW ^\tau- \tilde{\bW}^\tau \right\rVert^2\right)
    \end{align*}
    We choose $\rho^* = \frac{m_1 +m_2}{2} - \sqrt{\frac{(m_1 -m_2)^2}{4} + g}$, so that the assumption that \(m_1 m_2 > g\) implies that $\rho^*>0$. Furthermore, it is the smaller root of \((m_1 - \rho) (m_2 -\rho) = g\), so that \(\rho^* < m_1 \wedge m_2\). Finally, choose $\alpha = \alpha^* = \frac{L_{Vu}}{L_{uV}}$, so that $\frac{C_{\alpha^*}^2}{ 4 \alpha^*} = L_{Vu} L_{uV} = g$, which implies that the quadratic form 
    \begin{align*}
        (m_1 - \rho^*) \mathbb{E}\left\lVert u_\tau - \tilde{u}_\tau \right\rVert^2 - C_{\alpha^*} \left(\mathbb{E}\left\lVert u_\tau - \tilde{u}_\tau \right\rVert^2\right)^{\frac{1}{2}} \left( \mathbb{E}\left\lVert \bW ^\tau- \tilde{\bW}^\tau \right\rVert \right)^{\frac{1}{2}}+ \alpha^*(m_2 - \rho^*) \left\lVert \bW ^\tau- \tilde{\bW}^\tau \right\rVert^2 \geq 0.
    \end{align*}
   Thus \(\frac{d}{d\tau}\Phi_\tau \leq - 2 \rho^* \Phi.\) Assumption \eqref{ass:bouned-all} is a sufficient condition of following
   \begin{align}
    &\mathbb{E} \Phi_0^1 < \infty;\label{ass:addtional-ass1}\\
    &\mathbb{E} \int_0^T \left[  \left\vert  \frac{d}{d\tau} \left\lVert u_\tau -\tilde{u}_\tau\right\rVert^2\right\vert + \left\vert  \frac{d}{d\tau} \left\lVert W^\tau -\tilde{W}^\tau\right\rVert^2\right\vert\right] d \tau < \infty \text{ for every } T>0.\label{ass:addtional-ass2}
   \end{align}
   For \eqref{ass:addtional-ass1}, using the inequality \(\left\lVert z - \tilde{z}\right\rVert \le 2 \left\lVert z\right\rVert^2 + 2 \left\lVert \tilde{z}\right\rVert^2\),
    \begin{align*}
        \mathbb{E}\Phi_0^1 = \mathbb{E}\left[\left\lVert u_0 - \tilde{u}_0\right\rVert^2 + \left\lVert \bW^0 -\tilde{\bW}^0\right\rVert^2 \right] \le \mathbb{E}\left\lVert u_0\right\rVert^2 + \mathbb{E} \left\lVert  \tilde{u}_0\right\rVert^2 + \mathbb{E} \left\lVert \bW ^0\right\rVert^2 + \mathbb{E} \left\lVert \tilde{\bW}^0\right\rVert^2 < \infty.
    \end{align*}
    Similarly \eqref{ass:addtional-ass2} is bounded with \(2ab< a^2 + b^2\) as following 
    \begin{align*}
        &\left\lvert \frac{d}{d\tau} \left\lVert u_\tau- \tilde{u}_\tau\right\rVert^2\right\rvert + \left\lvert \frac{d}{d\tau}\left\lVert \bW^\tau -\tilde{\bW}^\tau \right\rVert^2\right\rvert \le \Phi^1_\tau + \left\lVert \nabla_u H(x, u_\tau,\bW^\tau)- \nabla_uH(x,\tilde{u}_\tau, \tilde{W}^\tau \right\rVert^2 \nonumber\\
        &+\left\lVert \bW^\tau - \tilde \bW^\tau\right\rVert^2+\left\lVert -\beta (\bW^\tau - \tilde{\bW}^\tau) + \cG(x,u_\tau, \bW^\tau) - \cG(x,\tilde{u}_\tau, \tilde{\bW}^\tau)\right\rVert^2.
    \end{align*} 
   With assumption \eqref{ass:addtional-ass1} and \eqref{ass:addtional-ass2} ,Gronwall’s inequality yields \(\Phi_\tau \leq e^{-2\rho^* \tau} \Phi_0\).

    Let \(\mathcal{L}_0 , \tilde{\mathcal{L}_0}\) denote the probability distribution of the processes \((u_0, \bW^0)\) and \((\tilde{u}_0, \tilde{\bW}^0)\), and denote as $\Pi(\mathcal{L}_0, \tilde{\mathcal{L}}_0)$ the set of coupling of two laws. From above with definition of Wasserstein-2 distance, 
    \begin{align*}
        \cW_{2, \alpha^*}^2 (\mathcal{L}_\tau, \tilde{\mathcal{L}_\tau }) \leq \mathbb{E}\left\lVert (u_\tau, \bW^\tau) - (\tilde{u}_\tau, \tilde{\bW}^\tau)\right\rVert_{\alpha^*}^2  \leq e^{-2 \rho^* \tau} \mathbb{E}\left\lVert (u_0, \bW^0) - (\tilde{u}_0, \tilde{\bW}^0) \right\rVert_{\alpha^*}^2.    \end{align*}
    Since this holds for every \(\pi_0 \in \Pi(\mathcal{L}_0, \tilde{\mathcal{L}}_0), \) therefore 
    \begin{align*}
        \cW^2_{2, \alpha^*}(\mathcal{L}_\tau, \tilde{\mathcal{L}}_\tau) \leq e^{-2 \rho^{*} \tau} \inf_{\pi_0 \in \Pi(\mathcal{L}_0, \tilde{\mathcal{L}}_0)} \mathbb{E}_{\pi_0 } \left\lVert (u_0, \bW^0) - (\tilde{u}_0, \tilde{\bW}^0) \right\rVert_{\alpha^*}^2 = e^{-2 \rho^{*} \tau} \cW _{2, \alpha^*}^2 (\mathcal{L}_0, \tilde{\mathcal{L}}_0),
    \end{align*}
    which yields the desired bound $\cW_{2, \alpha^*}(\mathcal{L}_\tau, \tilde{\mathcal{L}}_\tau) \leq e^{- \rho^{*} \tau} \cW _{2, \alpha^*} (\mathcal{L}_0, \tilde{\mathcal{L}}_0)$. 
\end{proof}

Theorem~\ref{thm:joint-contraction} compares two solutions of \eqref{eq:Z} driven by the same Brownian motion, and the contraction between the two requires Lipschitz conditions to hold for any pair of solutions. The next corollary compares the iterate (with $\lambda>0$) to the optimal control itself, in the spirit of the nonasymptotic bound for stochastic gradient Langevin dynamics in \citet{raginsky2017non}, which only requires \eqref{eq:H1}-\eqref{eq:H4} to hold locally around the optimal solution.
\begin{corollary}\label{cor:bias}
Fix $x\in\R^d$, let $(u_\tau,\bW^\tau)$ solve \eqref{eq:Z} with $\lambda$ replaced by a measurable, locally bounded temperature $\lambda(\tau)\ge0$, and $(u^*,\bW^*)$, $\bW^*=(v^*,v^*_x,v^*_{xx},p^*,q^*)$ be its stationary state, satisfying
\begin{equation}\label{eq:equilibrium}
\nabla_u\bH(x,u^*,\bW^*)=0,\qquad \beta\bW^*=\cG(x,u^*,\bW^*),
\end{equation}
Assume \eqref{ass:bouned-all}, and that for every $\tau\ge0$ the inequalities \eqref{eq:H1}--\eqref{eq:H4} hold with
\begin{align*}
&\eqref{eq:H1}:\ (u,\bW,\tilde\bW)=(u_\tau,\bW^\tau,\bW^*),\qquad\quad
\eqref{eq:H2}:\ (u,\tilde u,\bW)=(u_\tau,u^*,\bW^*),\\
&\eqref{eq:H3}:\ (u,\tilde u,\bW)=(u_\tau,u^*,\bW^\tau),\text{ and }(u,u^*,\bW^*)\quad\;\;
\eqref{eq:H4}:\ (u,\bW,\tilde\bW)=(u^*,\bW^\tau,\bW^*),
\end{align*}
with constants independent of $\tau$ satisfying \eqref{eq:smallgain} and $L_{uV}L_{Vu}>0$. Let $\alpha^*,\rho^*$ be as in \eqref{eq:alpha} and \eqref{eq:contraction}. Then $u^*$ is the unique maximizer of $\bH(x,\cdot,v^*_x,v^*_{xx})$, $v^*=V(x)$. If $\hat u = \{u_t = u^*,t\geq 0\} \in \mathcal A(x)$, then $\hat u$ is the optimal control. Furthermore, with $\Phi^*_\tau:=\|u_\tau-u^*\|^2+\alpha^*\|\bW^\tau-\bW^*\|^2$,
\begin{equation}\label{eq:bias-L2}
\E\|u_\tau(x)-u^*(x)\|^2+\alpha^*\,\E|v^\tau(x)-v^*(x)|^2
\le \E\Phi^*_\tau\le e^{-2\rho^*\tau}\,\E\Phi^*_0+2n\int_0^\tau e^{-2\rho^*(\tau-s)}\lambda(s)\,ds .
\end{equation}
\end{corollary}

\begin{proof}{Proof.}
By \eqref{eq:H3} with $(u,\tilde u,\bW)=(u,u^*,\bW^*)$ and \eqref{eq:equilibrium}, $\langle\nabla_u\bH(x,u,\bW^*),u-u^*\rangle\le-\kappa\|u-u^*\|^2$ for all $u$. Integrating along $u_t:=u^*+t(u-u^*)$,
\[
\bH(x,u,\bW^*)-\bH(x,u^*,\bW^*)=\int_0^1\frac1t\big\langle\nabla_u\bH(x,u_t,\bW^*),\,u_t-u^*\big\rangle dt\le-\frac{\kappa}{2}\|u-u^*\|^2.
\]
Thus $u^*$ is the unique maximiser, and Proposition~\ref{prop:limit-det} applies. 

The first inequality in \eqref{eq:bias-L2} holds by the definition of $\|\bW\|$. For the second, It\^o's formula and \eqref{eq:equilibrium} imply
\begin{align}
\Phi^*_\tau=\Phi^*_0&+2\int_0^\tau\big\langle\nabla_u\bH(x,u_s,\bW^s)-\nabla_u\bH(x,u^*,\bW^*),\,u_s-u^*\big\rangle ds\nonumber\\
&+2\alpha^*\int_0^\tau\big\langle-\beta(\bW^s-\bW^*)+\cG(x,u_s,\bW^s)-\cG(x,u^*,\bW^*),\,\bW^s-\bW^*\big\rangle ds\nonumber\\
&+2n\int_0^\tau\lambda(s)\,ds+2\int_0^\tau\sqrt{2\lambda(s)}\,\langle u_s-u^*,dB_s\rangle.\label{eq:Phi-ito}
\end{align}
By \eqref{ass:bouned-all} the stochastic integral is a true martingale. Then follow similar argument as for Theorem~\ref{thm:joint-contraction}, $\E\Phi^*_\tau\le\E\Phi^*_0-2\rho^*\int_0^\tau\E\Phi^*_s\,ds+2n\int_0^\tau\lambda(s)\,ds$, and Gronwall's lemma yields \eqref{eq:bias-L2}.
\end{proof}

\begin{remark}
Corollary \ref{cor:bias} implies that under the Lipschitz conditions in expectation, we can design annealed temperature to retrieve the optimal solution to the stochastic control problem. For a constant temperature $\lambda$ the last term in \eqref{eq:bias-L2} equals $\frac{n\lambda}{\rho^*}\big(1-e^{-2\rho^*\tau}\big)$. While for $\lambda(\tau)=\lambda_0e^{-c\tau}$ with $c\neq2\rho^*$ it equals $\frac{2n\lambda_0}{2\rho^*-c}\big(e^{-c\tau}-e^{-2\rho^*\tau}\big)$, so that $u_\tau(x)\to u^*(x)$ and $v^\tau(x)\to v^*(x)$ in $L^2$ at the rate $e^{-(c\wedge2\rho^*)\tau/2}$.

Since $v^\tau(x)=J(x,\mathbf u^\tau)\le V(x)$ by admissibility, Jensen's inequality and \eqref{eq:bias-L2} bound the expected excess value, $0\le V(x)-\E v^\tau(x)\le(\E\Phi^*_\tau/\alpha^*)^{1/2}$, the analogue of the expected excess risk in \citet{raginsky2017non}: at a constant temperature the price of exploration is of order $\sqrt{n\lambda/\rho^*}$, and it disappears under annealing. With $\lambda\equiv0$ and deterministic initial data, \eqref{eq:bias-L2} is the convergence $\|z_\tau-z^*\|_{\alpha^*}\le e^{-\rho^*\tau}\|z_0-z^*\|_{\alpha^*}$ of the deterministic iteration of Section~\ref{sec: classic} to the optimal pair, under \eqref{eq:H1}--\eqref{eq:H4}.
\end{remark}
\section{Numerical Examples}\label{sec: numerics}
In this section we implement the iteration developed in this paper to two example. The first one is the linear-quadratic model, which is the benchmark model most of reinforcement learning literature focus on. We use annealed temperature coefficient so that the iteration converges to the optimal solution of the stochastic control problem. 
\subsection{Linear-Quadratic Model}
In this section we verify the Lipschitz condition in expectation in \eqref{eq:H1}--\eqref{eq:H4} in a general linear-quadratic model for control iteration that starts from Gaussian initial values. Suppose the state \(X\) follows 
\begin{align*}
    dX_t = (AX_t + Bu_t) dt + \sigma dW_t,
\end{align*}
where \(X_t \in \mathbb{R}^d, u_t \in \mathbb{R}^n, A \in \mathbb{R}^{d \times d}, B \in \mathbb{R}^{d\times n}\), with quadratic running reward $f(x,u) = - x^\top Qx - u^\top R u$, nd \(Q \succeq 0, R \succeq 0.\)  The corresponding Hamiltonian 
\begin{align*}
    &\bH(x, u, v_x, v_{xx}) = -x^\top Q x - u^\top R u + (Ax + Bu)^T v_x + \frac{1}{2} Tr(\sigma^\top \sigma v_{xx}),\\
    &\nabla_u \bH(x,u ,W) = -2 Ru + B^\top v_x,
\end{align*}
and the control iteration is
\begin{align*}
    du_\tau = (-2R u_\tau(x) + B^\top v^\tau_x (x)) d\tau +\sqrt{2\lambda} dB_\tau.
\end{align*}
We choose an affine $u_0(x) = K_0 x + \xi_0$, where \(K_0 \in \mathbb{R}^{n \times d}\) is deterministic and \(\xi_0\) is Gaussian independent of Brownian motion \(B.\) The initial value \(v^0 = J(x,u^0)\), where $u^0_t = u_0$ for every $t\geq 0$ can be calculated explicitly,
\begin{align}
    v^0 &:= x^\top P_0 x + m_0^\top x + c_0,\\
    P_0 &= -\int_0^\infty e^{-\beta t}\, e^{(A+BK_0)^\top t}\,\big(Q + K_0^\top R K_0\big)\, e^{(A+BK_0)t}\,dt, \label{eq:P0}\\
m_0 &= 2\big[\beta I_d - (A+BK_0)^\top\big]^{-1}\big(P_0 B - K_0^\top R\big)\,\xi_0, \label{eq:m0}\\
c_0 &= \frac{-\xi_0^\top R\,\xi_0 + \xi_0^\top B^\top m_0 + \operatorname{Tr}\!\big(\sigma\sigma^\top P_0\big)}{\beta}. \label{eq:c0}
\end{align}
The dynamics \(\frac{d}{d\tau}v^\tau = -\beta v^\tau(x) +  \bH(x,u_\tau , v^\tau_x (x), v^\tau_{xx}(x)),\) implies that \(v^\tau\) is quadratic for every $\tau >0$, and can be written as
\begin{align*}
    v^\tau = x^\top P_\tau x + m^\top_\tau x + c_\tau,\quad v^\tau_x = 2 P_\tau x + m_\tau, \quad
    v^\tau_{xx} = 2P_\tau.
\end{align*}
Therefore, the control iterates remain linear in $x$, and can be written as $u_\tau(x) =  K_\tau x + \xi_\tau$, where
\begin{align}
    &du _\tau= (-2R K_\tau + 2 B^\top P_\tau)x d_\tau + (-2 R \xi_\tau + B^\top m_\tau) d\tau + \sqrt{2\lambda} dB_\tau;\\
    &dK_\tau = (-2 RK_\tau + 2B^\top P_\tau) d\tau;\label{eq:K_iteration}\\
    & d \xi_\tau = (-2 R \xi_\tau + B^\top m_\tau) d \tau + \sqrt{2\lambda} dB_\tau.\label{eq:xi_iteration}
\end{align} 
The value iteration implies the following dynamics in the parameters of $v^\tau$
\begin{align}
    &\frac{dP_\tau}{d \tau } =  - \beta P_\tau - Q - K_\tau^\top R K_\tau + (A+ BK_\tau)^\top P_\tau + P_\tau(A+BK_\tau);\label{eq:P_iteration}\\
    &\frac{dm_\tau}{d \tau} = \left[ -\beta I + (A + BK _\tau)^\top\right] m _\tau + (-2RK_\tau +2B^\top P_\tau)^\top \xi_\tau;\label{eq:m_iteration}\\
    &\frac{dc_\tau}{d\tau} = -\beta c_\tau - \xi_\tau^\top R \xi_\tau + \xi_\tau^\top B^\top m_\tau + Tr(\sigma \sigma^\top P_\tau).\label{eq:c_iteration}
\end{align}
where \(\xi_\tau^\top R\xi_\tau\) is quadratic Gaussian and \(c_\tau\) generally is not. In this case,
\begin{align*}
    \bW^\tau :=(v^\tau(x), v^\tau_x(x), v^\tau_xx(x) , p_\tau, q_\tau) =(x^\tau P_\tau x+ m_\tau^\top x+ c_\tau, 2 P_\tau x+ m_\tau, 2 P_\tau , K_\tau, 0).
\end{align*}
Therefore, the policy-value iteration is a closed system inside a finite-dimensional family $(K_\tau, \xi_\tau, P_\tau, m_\tau, c_\tau)$ governed by \eqref{eq:K_iteration} - \eqref{eq:c_iteration}.

\subsubsection{Convergence Conditions}

Next we find the Lipschitz constants $k(x)$, $L_{VV}(x)$, $L_{Vu}(x)$, and $L_{uV}(x)$ in Assumption \ref{assum:coupling-bounds} and verify the boundedness assumption \eqref{ass:bouned-all}, and the inequality \eqref{eq:smallgain}.

For fixed \(T,\) and \(0\le \tau \le T, K_\tau\) and \(P_\tau\) are deterministic. Since they are continuous function of \(\tau\),
\begin{align*}
    \sup_{0 \leq \tau \leq T} \left\lVert K_\tau \right\rVert < \infty, \quad \sup _{0\leq \tau \leq T} \left\lVert P_\tau \right\rVert < \infty.
\end{align*}
The iteration of \(\xi_\tau, m_\tau\) is a linear Gaussian system, with finite fourth moment, and 
\begin{align}\label{eq:bounded_fourth}
    \sup _{0\leq \tau\leq T} \mathbb{E} \left[\left\lVert \xi_\tau \right\rVert^4 + \left\lVert m_\tau\right\rVert ^4 \right] < \infty. 
\end{align}
\(\xi_\tau, m_\tau \in L^4\) implies that \(\xi_\tau^\top R \xi_\tau , \xi_\tau^\top B^\top m_\tau \in L^2\), and thus \(\sup_{t \leq T} \mathbb{E}\left\vert c_\tau \right\vert ^2 < \infty.\) Therefore
\begin{align*}
    \sup_{0\leq \tau \leq T}\mathbb{E}\left\lVert u_\tau\right\rVert ^2 = \sup_{0\leq \tau \leq T}\mathbb{E}\left\lVert K_\tau x+ \xi_\tau\right\rVert ^2 \leq 2\sup_{0\leq \tau \leq T}\left\lVert K_\tau\right\rVert ^2 \left\lVert x\right\rVert ^2 + 2 \sup_{0\leq \tau \leq T}\mathbb{E}\left\lVert \xi_\tau\right\rVert ^2 < \infty.
\end{align*}
Furthermore, with \(\left\lVert \bW^\tau\right\rVert^2 = \left\vert v \right\vert^2 + \left\lVert v_x\right\rVert ^2 + \left\lVert v_{xx} \right\rVert^2_F + \left\lVert p\right\rVert_F^2 + \left\lVert q\right\rVert _F^2\), $\sup_{0 \leq \tau \leq T } \mathbb{E} \left\lVert \bW^\tau\right\rVert ^2 < \infty$, and
\begin{align*}
    \sup_{0\leq \tau \leq T} \mathbb{E}\left\lVert \nabla_u \bH(x, u_\tau, \bW^\tau)\right\rVert^2 \leq& \sup _{0\leq \tau \leq T} \mathbb{E} \left\lVert -2Ru_\tau + B^\top v^\tau_x\right\rVert^2\\
    \leq& \sup _{0\leq \tau \leq T} 8 \left\lVert R\right\rVert ^2 \mathbb{E}\left\lVert u_\tau\right\rVert ^2 + 2 \left\lVert B\right\rVert ^2 \mathbb{E}\left\lVert v^\tau_x\right\rVert ^2 < \infty,
\end{align*}
Notice that
\begin{align}
    -\beta \bW^\tau + \mathcal{G}(x ,u_\tau, \bW^\tau) = \frac{d}{d\tau} \bW^\tau = (\dot{v}^\tau, \dot{v}^\tau_x, \dot{v}^\tau_{xx}, \dot{K}_\tau, 0),
\end{align}
with $\lVert \dot{K}_\tau \rVert \leq 2 \left\lVert R\right\rVert \left\lVert K_\tau\right\rVert + 2 \left\lVert B\right\rVert \left\lVert P_\tau \right\rVert$. Since \(K_\tau\) and \(P_\tau\) are bounded on \([0,T],\) \(\sup_{0\le \tau\le T} \mathbb{E} \lVert \dot{K}_\tau \rVert < \infty,\) and \(\sup_{0\le \tau \le T} \mathbb{E} \lVert \dot{K}_\tau \rVert ^2 < \infty.\) Furthermore, 
\begin{align*}
    \lVert \dot{P}_\tau \rVert  \le \beta \left\lVert P_\tau \right\rVert + \left\lVert Q\right\rVert  + \left\lVert R\right\rVert \left\lVert K_\tau \right\rVert ^2 + 2(\left\lVert A \right\rVert + \left\lVert B\right\rVert \left\lVert K_\tau \right\rVert) \left\lVert P_\tau \right\rVert,
\end{align*}
where every term on the right is bounded on \([0,T]\), and thus \(\sup_{0 \le \tau \le T}  \left\lVert \dot{P}_\tau\right\rVert  < \infty.\) Similarly,
\begin{align}
    \lVert \dot{m}_\tau\rVert \le 2 \left\lVert -\beta I + (A+ BK_\tau)^\top \right\rVert^2 \left\lVert m_\tau\right\rVert^2 + 2 \left\lVert -2RK_\tau+ 2B^\top P_\tau\right\rVert  ^2 \left\lVert \xi_\tau\right\rVert^2,
\end{align}
and \eqref{eq:bounded_fourth} implies that \(\sup_{0 \le \tau \le T} \mathbb{E}  \left\lVert \dot{m}_\tau\right\rVert^2 < \infty.\) Finally
\begin{align}
    \left\vert  \dot{c}_\tau \right\vert^2 \le 4 \left[ \beta^2 \left\vert   c_\tau \right\vert^2  + \left\vert \xi_\tau^\top R \xi_\tau \right\vert^2 + \left\vert  \xi_\tau^\top B^\top m_\tau \right\vert^2 + \left\vert  Tr(\sigma \sigma^\top P_\tau) \right\vert^2\right].
\end{align}
Since \(\xi_\tau, m_\tau \in L^4\), Cauchy-Schwarz inequality implies that \(\sup_{0 \le \tau \le T} \mathbb{E} \left\vert  \dot{c}_\tau \right\vert^2  < \infty.\)
With the boundedness from above we conclude that $\sup_{0 \le \tau \le T } \mathbb{E} \left\lVert -\beta \bW^\tau + \mathcal{G}(x,u_\tau, \bW^\tau) \right\rVert^2 < \infty$.

For the Lipschitz condition in expectation, 
\begin{align*}
    &\mathbb{E}\left\langle\nabla_u\bH(x,u,\bW)-\nabla_u\bH(x,\tilde u,\bW),\,u-\tilde u\right\rangle =  -2 (u - \tilde{u})^\top R( u - \tilde{u}) \le -2 \lambda_{min}(R) \left\lVert u -\tilde{u} \right\rVert^2;\\
    &\mathbb{E}\big\|\nabla_u\bH(x,u,\bW)-\nabla_u\bH(x,u,\tilde \bW)\big\|^2 =\mathbb{E} \lVert B^\top(v_x - \tilde{v}_x)\rVert^2 \le \left\lVert B\right\rVert^2 \mathbb{E}\left\lVert \bW - \tilde{\bW} \right\rVert^2.
\end{align*}
Therefore, \(k(x) = 2\lambda_{\min}(R), L_{Vu}(x) = \left\lVert B\right\rVert.\) For \eqref{eq:H1} and \eqref{eq:H2}, 
\begin{align*}
    \mathcal G(x,u,\bW)=\begin{pmatrix}
-x^\top Q x - u^\top R u + (Ax + Bu)^T v_x + \frac{1}{2} Tr(\sigma^\top \sigma v_{xx})
\\[2mm]
-2Qx + A^\top v_x + v_{xx}(Ax+Bu)
+ p^\top(-2Ru + B^\top v_x)
\\[2mm]
-2Q + A^\top v_{xx} + v_{xx}A
+ p^\top B^\top v_{xx}
+ v_{xx}Bp
- 2p^\top Rp
\\[2mm]
(\beta I - 2R)p + B^\top v_{xx}
\\[2mm]
0
\end{pmatrix}.
\end{align*}
and 
\begin{align*}
G(x,u, \bW)-G(x,\widetilde u, \bW)
=
\begin{pmatrix}
(u-\widetilde u)^\top
\left[
B^\top  v_x
-
R(u+\widetilde u)
\right]
\\[2mm]
\left[
 v_{xx}B
-
2 K^\top R
\right]
(u-\widetilde u)
\\[2mm]
0
\\
0
\\
0
\end{pmatrix}.
\end{align*}
Define \(\bar{P} := \max \{\sup_{0 \le \tau \le T} \left\lVert P_\tau\right\rVert, \sup_{0 \le \tau \le T} \lVert \tilde{P}_\tau\rVert \}, \bar{K} := \max \{\sup_{0 \le \tau \le T} \left\lVert K_\tau\right\rVert, \sup_{0 \le \tau \le T} \lVert \tilde{K}_\tau\rVert \},\) then with H\"older's inequality and treating \(u - \tilde{u}\) as gussian difference,
\begin{align*}
    \left\lVert B^\top v_x - R(u + \tilde{u})^\top (u -\tilde{u})\right\rVert_{L^2} &\le \left\lVert B^\top v_x - R(u + \tilde{u})\right\rVert_{L^4} \left\lVert u - \tilde{u} \right\rVert_{L^4} 
 \le  \left\lVert B^\top v_x - R(u + \tilde{u})\right\rVert_{L^4} 3^{1/4} \left\lVert u - \tilde{u} \right\rVert_{L^2}.
\end{align*}
Since \(v_{xx} = 2P \) and \(p = K\) are deterministic 
\begin{align*}
    \left\lVert v_{xx}B - 2 p^\top R\right\rVert \le 2 \left\lVert B \right\rVert \bar{P} + 2 \left\lVert R \right\rVert \bar{K}.
\end{align*}
Combining them together implies 
\begin{align}
    L_{uV} = \left[ (2 \left\lVert B\right\rVert) \bar{P} + 2 \left\lVert R\right\rVert \bar{K})^2 + \sqrt{3} \sup_{0 \le \tau \le T} \left\lVert B^\top v^\tau_x - R (u_\tau + \tilde{u}_\tau)\right\rVert^2_{L^4}\right]^{1/2}.\label{eq:LQ_LuV}
\end{align}
On the other hand,
\begin{align*}
G(x,u,\bW)-G(x,u,\widetilde \bW)= 
&\begin{pmatrix}
(Ax+Bu)^\top(v_x-\widetilde v_x)
+\frac12\operatorname{Tr}\!\left[\sigma\sigma^\top(v_{xx}-\widetilde v_{xx})\right]
\\[3mm]
\begin{aligned}[t]
&A^\top(v_x-\widetilde v_x)
+(v_{xx}-\widetilde v_{xx})(Ax+Bu)\\
&\quad+p^\top(-2Ru+B^\top v_x)
-\widetilde p^\top(-2Ru+B^\top\widetilde v_x)
\end{aligned}
\\[3mm]
\begin{aligned}[t]
&A^\top(v_{xx}-\widetilde v_{xx})
+(v_{xx}-\widetilde v_{xx})A\\
&\quad+p^\top B^\top v_{xx}+v_{xx}Bp
-\widetilde p^\top B^\top\widetilde v_{xx}
-\widetilde v_{xx}B\widetilde p\\
&\quad-2\bigl(p^\top Rp-\widetilde p^\top R\widetilde p\bigr)
\end{aligned}
\\[3mm]
(\beta I-2R)(p-\widetilde p)
+B^\top(v_{xx}-\widetilde v_{xx})
\\[3mm]
0
\end{pmatrix}.
\end{align*}
For inner product with respect to norm of \(\bW,\) using Cauchy-Schwarz inequality and the Gaussian fourth-moment bound,
\begin{align}
L_{VV}=&
\max\left\{
0,
2\lambda_{\max}\!\left(\frac{A+A^\top}{2}\right)
+2\|B\|\bar{K},
\beta-2\lambda_{\min}(R)
\right\} \nonumber \\
&+\frac14\|\sigma\sigma^\top\|_F
+\frac{1}{2}\|B\|
+2
\left(\|B\|\bar{P}+\|R\|\bar{K}\right) \nonumber\\
&+\sup_{0 \le \tau \le T}\left[
\frac{\sqrt{3}+1}{2}\|Ax+Bu_\tau\|_{L^2}
+\frac{\|-2Ru_\tau+B^\top v_x^\tau\|_{L^2}}
{2}
\right].\label{eq:LQ_LVV}
\end{align}
Notice that these Lipschitz constants depend on $\bar P$ and $\bar K$, and thus on $T$. In numerical implementations, the rounds of iteration are always finite. Thus \eqref{eq:smallgain} is satisfied for sufficiently large $\beta$ and leads to exponential convergence to optimal policy and value iteration.
\subsubsection{Numerical Implementation}
we use two sets of parameters for numerical experiments. The first set is 
\begin{align}\label{model:LQ_scalar}
    A = -1, \quad B = 1, \quad Q = 1, \quad R =1, \quad \beta =0.05, \quad \sigma =0.2.
\end{align}
The scalar model is a special case of classical discounted LQ problem studied in \citet{wang2020reinforcement}. With \(\beta > 2A\) and $k_2 =  \frac{R(\beta-2A)-\sqrt{R^2(\beta-2A)^2+4RB^2Q}}{B^2}$, the optimal value and control are
\begin{align*}
    u^*(x) = K^*x, \quad v^*(x) = P^* x^2 + c^*, \quad P^* = \frac{k_2}{2}, \quad K^* = \frac{BP^*}{R}, \quad c^* = \frac{\sigma^2 P^*}{\beta}.
\end{align*}
We initialize the iteration with $u_0 = K_0 x+ \xi_0$, where \(K_0  = -10, \xi_0 \sim N(0,0.5^2)\). The coefficients of the initial value
function \(v^0 = x^\top P_0 x + m_0^\top x + c_0\) are (notice that $K_0 <0$ and $\beta > 2(A+BK_0)$)
\begin{align*}
    P_0 = -\frac{Q + RK_0^2}{\beta - 2(A + BK_0)}, \quad
    m_0 = \frac{2\,(BP_0 - RK_0)}{\beta - (A + BK_0)}\,\xi_0, \quad 
c_0 = \frac{-R\xi_0^2 + B\xi_0 m_0 + \sigma^2 P_0}{\beta}. 
\end{align*} 

We iterate over \(\tau \in [0,300]\) with step size \(\Delta \tau = 0.002\). Figure \ref{fig:LQ-profiles} shows the convergence of our iterates to the optimal control and value function. Notice that after $\tau = 2$, the value function moves in a parallel fashion. Thus its space derivatives $v_x$ and $v_{xx}$ converge fast, and so is the control variable $u_\tau$, whose dynamics only rely on the derivatives.

\begin{figure}[H]
\FIGURE
{\subfloat[Policy iteration ...\label{fig:LQ-policy}]{%
    \includegraphics[width=0.49\textwidth]{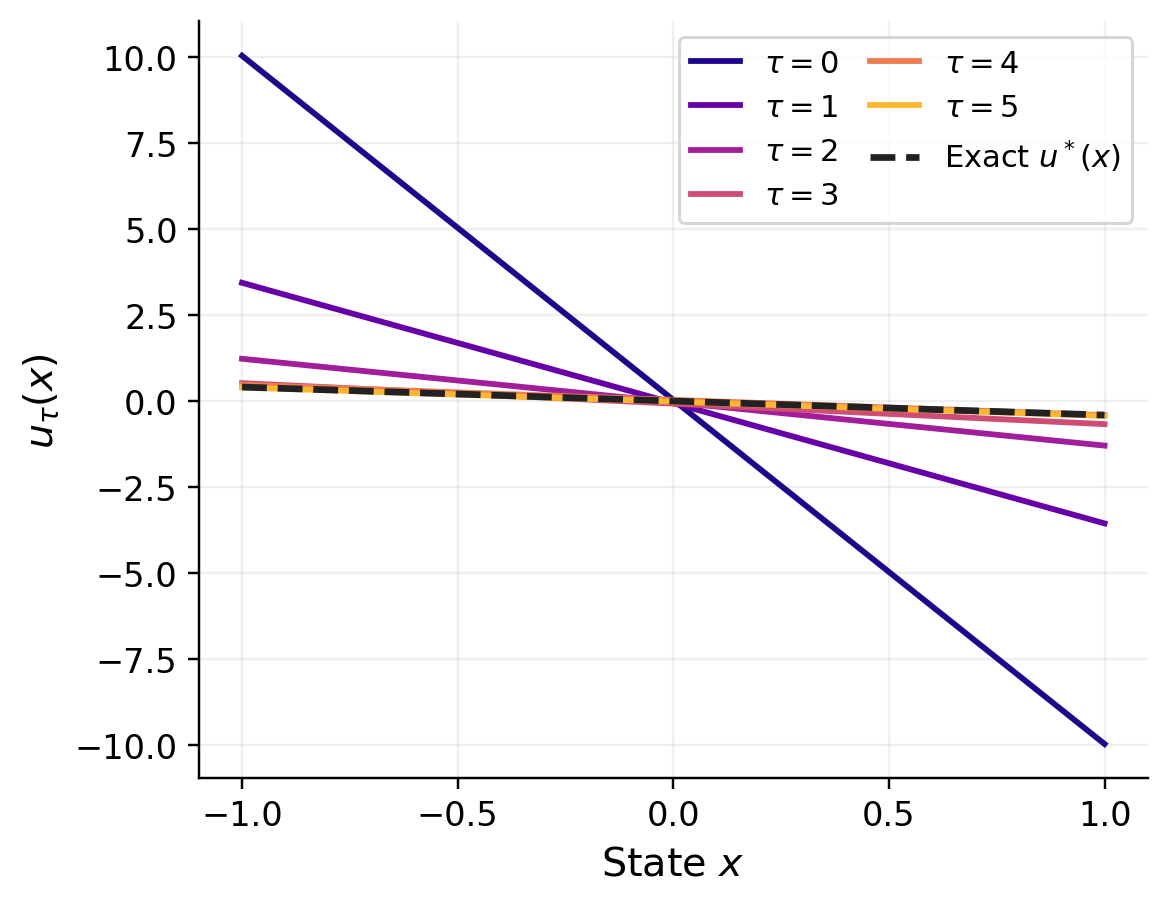}}\hfill
  \subfloat[Value iteration ...\label{fig:LQ-value}]{%
    \includegraphics[width=0.49\textwidth]{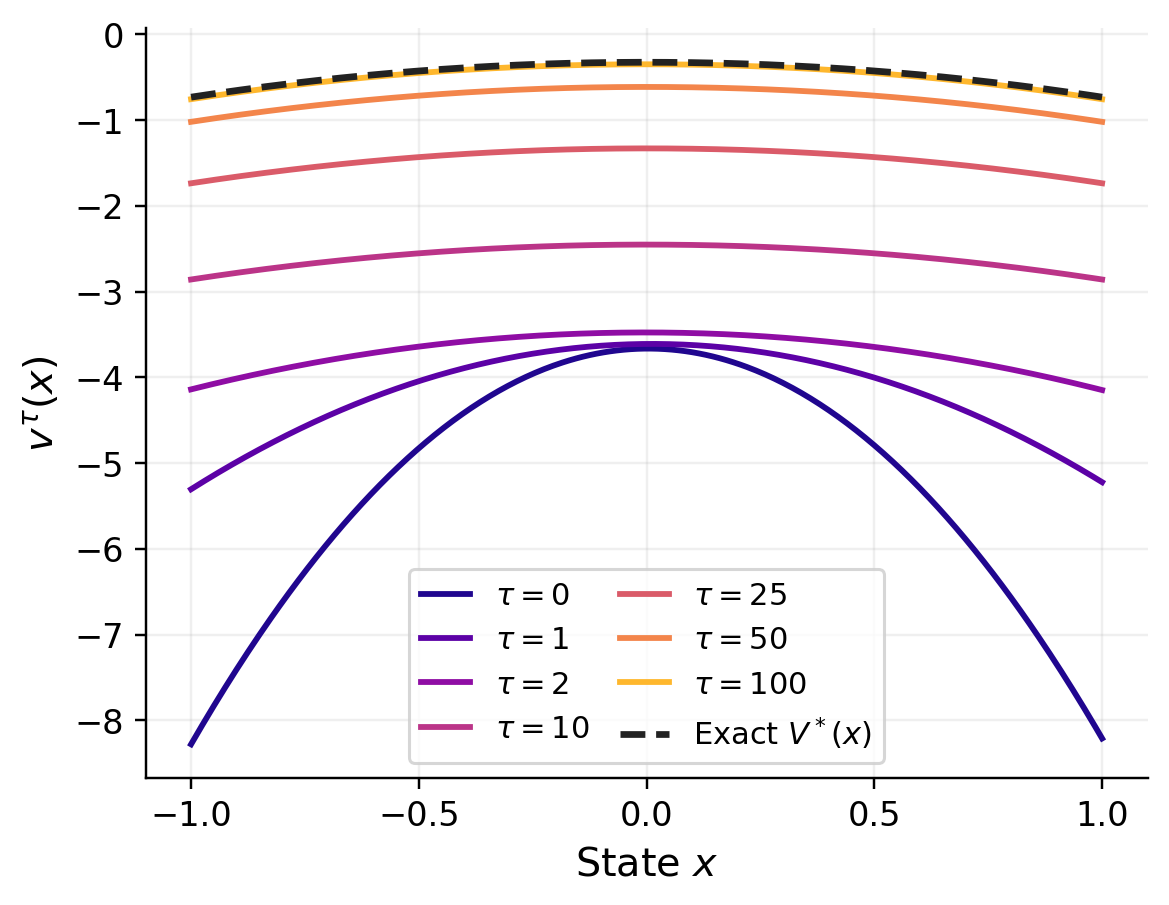}}}
{One sample path of policy and value iteration ... \label{fig:LQ-profiles}}
{}
\end{figure}
Figure \ref{fig:exponential-convergence} compares the iterated policy and value with optimal policy and value. It plots the mean weighted squared distance between them,
\begin{align*}
    \Phi_\tau^* = (u_\tau- u^*)^2 + (v^\tau -v^*)^2 +(\partial_x v^\tau - \partial_x v^*)^2 + 4 (P_\tau - P^*)^2 +(K_\tau -K ^*)^2,
\end{align*}
estimated from \(2,000\) sample paths, on a logarithmic scale. The distance decays exponentially. 
\begin{figure}[H]
\FIGURE
{\includegraphics[width=0.5\textwidth]{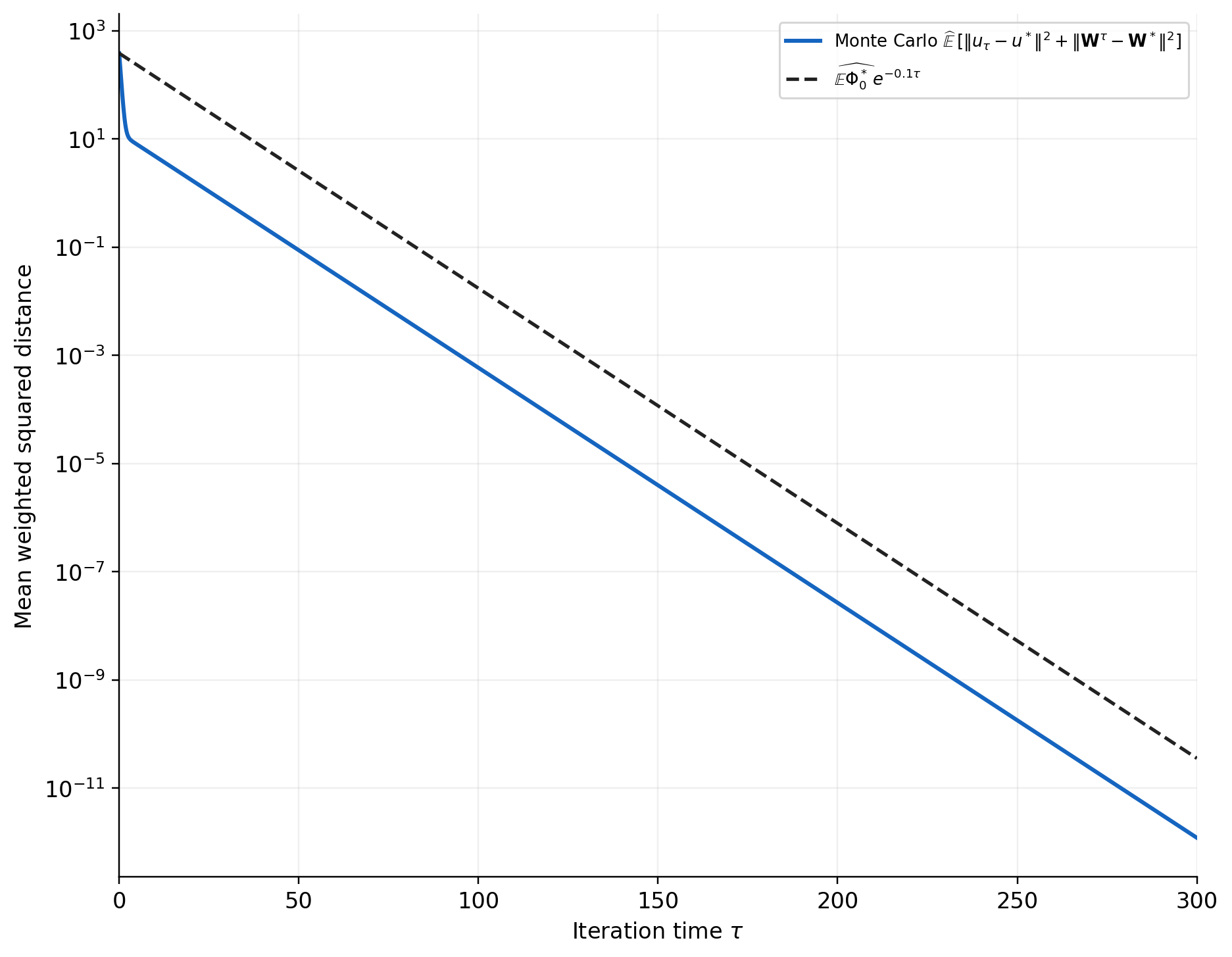}}
{Empirical convergence to the optimum. The blue curve shows the mean weighted squared error estimated using 2,000 Monte Carlo sample paths. The dashed line represents the reference decay \(\widehat{\mathbb E\Phi_{0}^*}e^{-0.1\tau}\). \label{fig:exponential-convergence}}
{}
\end{figure}
The second experiment demonstrates the case where the inequality \eqref{eq:smallgain} holds: 
\begin{align*}
A = -0.01, \quad B = 0.01, \quad Q = 0.01, \quad R =1, \quad \beta =0.05, \quad \sigma =0.1.
\end{align*}
From Riccati equation, \(K^*=-0.00143, P^*= -0.1428, c^*=-0.0286.\) In the following table we calculate the constants in \eqref{eq:H1} - \eqref{eq:H4} from Assumption \ref{assum:coupling-bounds} in each \(\tau \) and checks the inequality condition on the interval \([\tau_0, T]\) where $T = 300$ is the terminal round of iteration. Note that once the inequality is satisfied, it holds for all latter \(\tau\), which guarantees the exponential convergence. 

\begin{table}[H]
{Numerical check of inequality \eqref{eq:smallgain} at $x=1$, $T=300$, $k=2$, $L_{Vu}=0.01$, $\beta=0.05$ .\label{tab:LQ-contraction-example_2}}
{\begin{tabular}{@{}rrrrrrc@{}}
\toprule
$\tau_0$ & $L_{VV}$ & $L_{uV}$ & $k(\beta-L_{VV})$ & $L_{uV}L_{Vu}$ & $\rho_{\mathrm{calc}}$ & $k(\beta-L_{VV})>L_{uV}L_{Vu}$\\
\midrule
 90 & 0.0565 & 0.0352 & $-0.0131$ & 0.0004 & $-0.0067$ & FAIL\\
 92 & 0.0527 & 0.0314 & $-0.0053$ & 0.0003 & $-0.0028$ & FAIL\\
 94 & 0.0493 & 0.0280 & 0.0014 & 0.0003 & 0.0006 & PASS\\
 96 & 0.0464 & 0.0251 & 0.0073 & 0.0003 & 0.0035 & PASS\\
 98 & 0.0438 & 0.0226 & 0.0124 & 0.0002 & 0.0061 & PASS\\
100 & 0.0416 & 0.0203 & 0.0168 & 0.0002 & 0.0083 & PASS\\
125 & 0.0294 & 0.0082 & 0.0411 & 0.0001 & 0.0205 & PASS\\
150 & 0.0273 & 0.0062 & 0.0453 & 0.0001 & 0.0226 & PASS\\
200 & 0.0269 & 0.0057 & 0.0462 & 0.0001 & 0.0231 & PASS\\
250 & 0.0269 & 0.0057 & 0.0462 & 0.0001 & 0.0231 & PASS\\
\bottomrule
\end{tabular}}
{}
\end{table}

It also illustrates the fact that \eqref{eq:H1} - \eqref{eq:H4} and \eqref{eq:smallgain} are only sufficient conditions on the exponential convergence - in the early stage of the iteration (and also whole iteration in the experiment with the first set of parameters), \eqref{eq:smallgain} are violated, yet Figure \ref{fig:LQ}(a) shows the exponential convergence in joint contraction by defining (c.f. \eqref{eq:alpha} with $\alpha^* = 1$)
\begin{align*}
    \Phi_\tau = (u_\tau- \tilde{u}_\tau )^2 + (v_\tau -\tilde{v}^\tau )^2 +(\partial_x v^\tau - \partial_x \tilde{v}^\tau)^2 + 4 (P_\tau - \tilde{P}_\tau)^2 +(K_\tau -\tilde{K}_\tau )^2,
\end{align*}
and Figure \ref{fig:LQ}(b) towards the optimal solution on logarithmic scale. 
\begin{figure}[H]
\FIGURE
{\subfloat[Joint contraction distance.\label{fig:LQ-joint-contraction-2}]{%
    \includegraphics[width=0.49\textwidth]{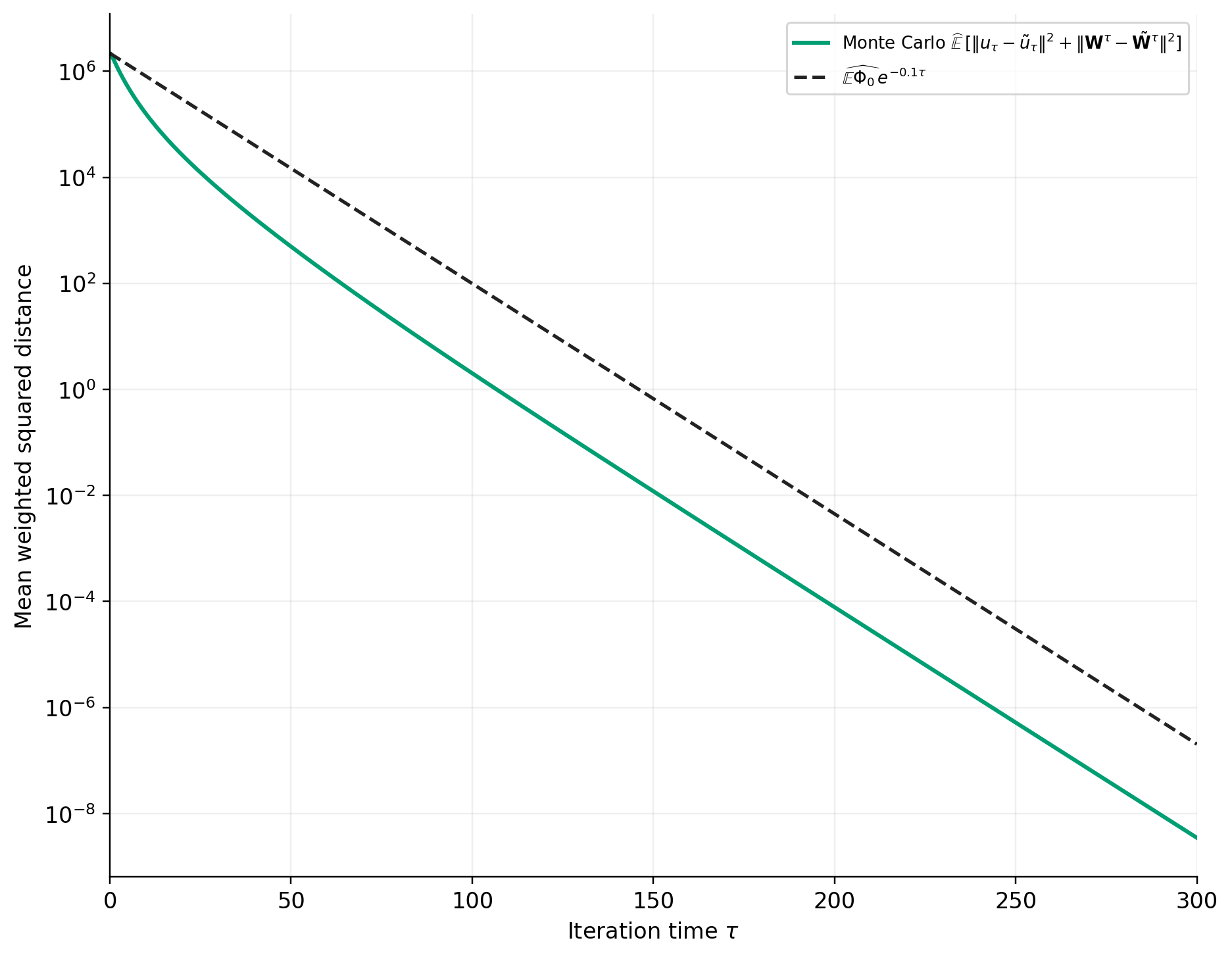}}\hfill
  \subfloat[Distance to the optimum.\label{fig:LQ-exponential-convergence-2}]{%
    \includegraphics[width=0.49\textwidth]{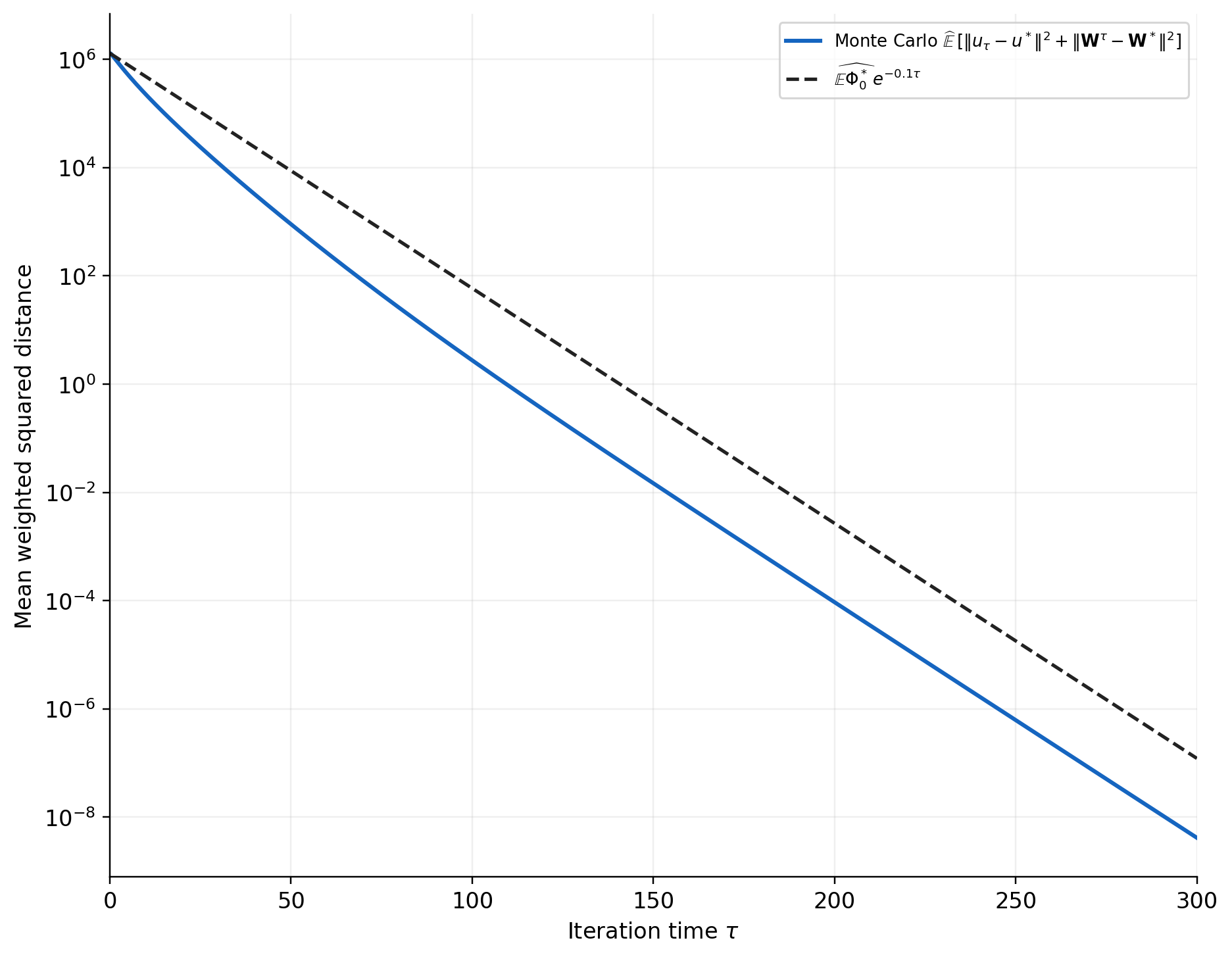}}}
{Empirical joint contraction and convergence to the optimum. Two iterations with initial policies $K_0=-10$ and $\tilde K_0=-20$ share the same initial $\xi_0\sim\mathcal N(0,0.5^2)$ and Brownian motion, with $\lambda(\tau)=0.04e^{-0.5\tau}$. (a) The green curve show the joint contraction distance. (b) The blue curve shows distance to optimal. Both curves are estimated using 2,000 Monte Carlo sample paths. The dashed line represents the reference decay \(\widehat{\mathbb E\Phi_{0}}e^{-0.1\tau}\)  and \(\widehat{\mathbb E\Phi_{0}^*}e^{-0.1\tau}\).\label{fig:LQ}} 
{}
\end{figure}
\subsection{Optimal Consumption and Investment in Incomplete Markets}\label{sec:power-utility-eg}
The second example is a classical yet notoriously open problem in mathematical finance -- optimal investment and consumption in incomplete market factor model. The associated HJB equation does not admit closed-form solutions in the power utility setting, and the extant theoretical results in approximate solutions all rely on delicate mathematical treatment, e.g. the closed-form solution in the completed market in \citet{karatzas1991brownian,he1991consumption} and market contraction in \citet{GW2020}, or convergent approximations by sub- and super-solution argument  \citet{gutekunst2026optimalinvestmentconsumptionstochastic}, which works only in one-dimensional factor model. \citet{bick2013solving} works in general multi-dimensional model where the dual bound results from the deterministic market price of risk, and thus does not provide theoretical guarantee of the approximation accuracy. Our method provides easily implementable alternative which converge to the optimal solution -- starting from any initial point, the approximation is generated by following the coupled iteration dynamics with elementary calculation each step. 

The model setup follows \citet{GW2020}, where the market consists of $n+1$ assets \((S_t^0, S_t^1, \dots, S_t^n)_{t\geq 0}\) with \(S^0\) being the risk-free asset. The risk free rate and the coefficients in the dynamics of $n$ risky assets are functions of \(k\) states variable \(Y = (Y_t^ i) ^{1\leq i\leq k}_{t\geq 0}\) such that \(Y_t \in E\) a.s. for all \(t\geq 0,\) where \(E \subset \mathbb{R}^k\) is open connected set. They follow the joint dynamics
\begin{align*}
    &\frac{dS_t^0}{S_t^0} = r(Y_t)dt;\\
    &\frac{dS_t^i}{S_t^i} = r(Y_t)dt + dR_t^i, \quad 1\leq i\leq n; \nonumber\\
    &dR_t^i = \mu _i(Y_t)dt + \sum\limits_{j=1}^{n} \sigma _{ij}(Y_t)dZ_t^j, \quad 1\leq i\leq n;\nonumber\\
    &dY_t^i = b_i(Y_t )dt + \sum\limits_{j=1}^{k} a_{ij} (Y_t) dW_t^j, \quad 1\leq i\leq k; \nonumber\\
    &d\langle R^i, Y^j\rangle_t = \Upsilon _{ij} (Y_t)dt, \quad 1\leq i\leq n,1\leq j\leq k,
\end{align*}  
where \(Z = (Z_j^i)_{t\geq 0} ^{1\leq i\leq n}\) and \(W = (W_t^i)_{t\geq 0}^{1\leq i\leq k}\) are multivariate standard Brownian motions with instantaneous cross correlation matrix \(\rho(Y_t) \in \mathbb{R}^{n \times k},\) and denote as $\mathbb{F} = \{\mathcal F_t\}_{t\geq 0}$ is the filtration generated by $Z$ and $W$. \(R = (R_t)_{t\geq 0}\) is the cumulative return process and the instantaneous covariance matrices are characterized by their quadratic variations and covariations:
\[\Sigma  = \sigma \sigma ^\top  = d\langle R, R \rangle_t / dt, \quad  A = aa^\top  = d \langle Y,Y \rangle_t /dt, \quad \Upsilon = \sigma \rho a^\top = d \langle R ,Y \rangle_t /dt. \]
Given an initial wealth $x$, an agent trades according to an $n$-dimensional proportional portfolio strategy $\pi = (\pi_t)_{t\geq 0}$ and consumes at a continuous rate $c_t = \ell_t X_t$. The wealth process $X^{\pi,\ell}$ follows
\begin{align*}
    dX_t^{\pi, \ell} = X_t^{\pi, \ell} (rdt + \pi^\top \mu(Y_t)  -\ell_t) dt + X_t^{\pi, \ell} \pi^\top_t \sigma(Y_t) dZ_t.
\end{align*}
To ensure the wealth and consumption processes remain well-defined and non-negative, we restrict our attention to the class of admissible control strategies.
\begin{definition}
The set of admissible investment and consumption strategies, denoted by $\mathcal{A}(x,y)$, consists of all pairs of $\mathbb{F}$-adapted processes $(\pi, \ell)$ such that starting from initial values $(X_0,Y_0) = (x,y)$, $\pi$ is integrable with respect to the return process $R$, and the corresponding wealth process $X^{\pi, \ell}$ and consumption proportion $\ell$ satisfy $X_t^{\pi, \ell} > 0$ and $\ell_t \ge 0$ $\mathbb{P}$-almost surely for all $t \ge 0$.
\end{definition}
The agent's objective is to maximize the expected cumulative discounted power utility from consumption over an infinite time horizon:
\begin{equation}
V(x,y) = \sup_{(\pi, \ell) \in \mathcal{A}(x,y)} \mathbb{E} \left[ \int_0^\infty e^{-\beta t} \frac{(\ell_t X_t^{\pi, \ell})^{1-\gamma}}{1-\gamma} dt \right].
\end{equation}
We assume that the problem is well-posed:
\begin{assumption}\label{assumption-1-example}
     $V(x,y) < \infty$.
\end{assumption}

The HJB equation corresponding to the above stochastic control problem is
\begin{equation}
-\beta v(x,y) + \sup_{\pi,\ell} \bH(x,y,\pi,\ell, v_x, v_{xx}, \nabla_y v, \nabla_y v_x) = 0,
\end{equation} 
where the Hamiltonian $\bH$
\begin{align*}
H(&x,y,\pi,\ell, v_x, v_{xx}, \nabla_y v, \nabla_y v_x) \nonumber \\
&= \frac{(\ell x)^{1-\gamma}}{1-\gamma} + x(r + \pi^\top \mu - \ell)v_x + b^\top \nabla_y v + \frac{1}{2}x^2 \pi^\top \Sigma \pi v_{xx} + \frac{1}{2}\operatorname{Tr}(A D_{yy}^2 v) + x \pi^\top \Upsilon \nabla_y v_x.
\end{align*}

To implement the continuous policy iteration algorithm for $(\pi, \ell)$, we first compute the gradients of the Hamiltonian with respect to the control variables:
\begin{align*}
\nabla_\pi H &= x \mu v_x + x^2 \Sigma \pi v_{xx} + x \Upsilon \nabla_y v_x,\\
\nabla_\ell H &= x^{1-\gamma} \ell^{-\gamma} - x v_x.
\end{align*}

For each state $(x,y)$, given initial policies $\pi_0^{x,y}$ and $\ell_0^{x,y}$, the continuous updates driven by Langevin-type dynamics over iteration time $\tau$ are
\begin{align*}
d\pi_\tau^{x,y} &= (x \mu v_x^\tau + x^2 \Sigma \pi_\tau^{x,y} v_{xx}^\tau + x \Upsilon \nabla_y v_x^\tau) d\tau + \sqrt{2\lambda} dB_\tau^\pi \\
d\ell_\tau^{x,y} &= (x^{1-\gamma} (\ell_\tau^{x,y})^{-\gamma} - x v_x^\tau) d\tau + \sqrt{2\lambda} dB_\tau^\ell.
\end{align*}
The value iteration driven by continuous dynamics is given by the following system:
\begin{align*}
dv^\tau &= \left(-\beta v^\tau + H(x,y,\pi_\tau, \ell_\tau, v_x^\tau, v_{xx}^\tau, \nabla_y v^\tau, \nabla_y v_x^\tau)\right) d\tau, \\
dv_x^\tau &= \left(-\beta v_x^\tau + \nabla_x\left(H(x,y,\pi_\tau, \ell_\tau, v_x^\tau, v_{xx}^\tau, \nabla_y v^\tau, \nabla_y v_x^\tau)\right)\right) d\tau, \\
dv_{xx}^\tau &= \left(-\beta v_{xx}^\tau + \nabla_{xx}\left(H(x,y,\pi_\tau, \ell_\tau, v_x^\tau, v_{xx}^\tau, \nabla_y v^\tau, \nabla_y v_x^\tau)\right)\right) d\tau, \\
d\nabla_y v^\tau &= \left(-\beta \nabla_y v^\tau + \nabla_y H(x,y,\pi_\tau, \ell_\tau, v_x^\tau, v_{xx}^\tau, \nabla_y v^\tau, \nabla_y v_x^\tau)\right) d\tau, \\
d\nabla_y v_x^\tau &= \left(-\beta \nabla_y v_x^\tau + \nabla_y\left(\nabla_x\left(H(x,y,\pi_\tau, \ell_\tau, v_x^\tau, v_{xx}^\tau, \nabla_y v^\tau, \nabla_y v_x^\tau)\right)\right)\right) d\tau.
\end{align*}

To simplify the numerical implementation, we take advantage of the homogeneity of the power utility and the linearity of the dynamics of the wealth process, and rewrite $v^\tau$ (the value corresponding to the control strategy defined in \eqref{eq:u-reverse}) as
\[
v^\tau(t,x,y)=\frac{x^{1-\gamma}}{1-\gamma}h^\tau(t,y).
\]
Differentiating with respect to the state variables yields
\[
\begin{aligned}
v_x^\tau(t,x,y)
    &=x^{-\gamma}h^\tau(t,y),\qquad v_{xx}^\tau(t,x,y)=-\gamma x^{-\gamma-1}h^\tau(t,y),\\
\nabla_y v^\tau(t,x,y)
    &=\frac{x^{1-\gamma}}{1-\gamma}\nabla_y h^\tau(t,y),\qquad 
\nabla_y v_x^\tau(t,x,y)
    =x^{-\gamma}\nabla_y h^\tau(t,y),\\
\end{aligned}
\]

Thus, only the reduced function $h^\tau$ needs to be stored and its \(y\) gradient can be approximated using finite differences. Next we verify the concavity of $\bH^\tau$ with respect to the control variables $\pi$ and $\ell$, so that the iteration converges even with temperature $\lambda = 0$.
\begin{lemma}
$\bH^\tau = \bH(x,y,\pi,l,v^\tau_x,v^\tau_{xx},\nabla_y v,\nabla_y v_{x})$, is strictly concave with respect to the control variables $\pi$ and $\ell$.
\end{lemma}

\begin{proof}{Proof.}
Since $\pi$ and $\ell$ are separated in $\bH^\tau$, to establish the strict joint concavity, we only need to check the second-order derivatives of $\bH^\tau$ with respect to  $\pi$ and $\ell$:
\begin{align}
\nabla^2_\ell H^\tau &= -\gamma x^{1-\gamma}\ell^{-\gamma-1}, \quad \nabla^2_\pi H^\tau = x^2 v_{xx}^\tau \Sigma \label{eq:H_pi_pi}.
\end{align}
Given that the risk aversion parameter $\gamma > 0$, the initial wealth $x > 0$, and the consumption proportion $\ell > 0$, it follows that $\frac{\partial^2 H^\tau}{\partial \ell^2} < 0$. This implies that $H^\tau$ is strictly concave with respect to $\ell$. For the negative definiteness of the Hessian $\nabla^2_\pi H^\tau$, it suffices to show that $v_{xx}^\tau < 0$.

Let $x_1$ and $x_2$ be two distinct initial wealth levels ($x_1 \neq x_2$), and $u^{1,\tau} = \{(\pi_t^{1,\tau}, \ell_t^{1,\tau})\}_{t\geq 0}$ and $u^{2,\tau} = \{(\pi_t^{2,\tau}, \ell_t^{2,\tau})\}_{t\geq 0}$ be two admissible control strategies associated with  initial $(x_1,y)$ and $(x_2,y)$, respectively under our iteration scheme. Denote as $X_t^{1,\tau}$ and $X_t^{2,\tau}$ the corresponding wealth process, alongside the consumption processes:
\begin{equation}
c_t^{1,\tau} = \ell_t^{1,\tau} X_t^{1,\tau} \quad \text{and} \quad c_t^{2,\tau} = \ell_t^{2,\tau} X_t^{2,\tau}.
\end{equation}
With $\lambda \in (0,1)$, let $x^\lambda = \lambda x_1 + (1-\lambda)x_2$, and $X_t^{\lambda,\tau} = \lambda X_t^{1,\tau} + (1-\lambda)X_t^{2,\tau}$ and define
\begin{equation}
\pi_t^{\lambda,\tau} = \frac{\lambda X_t^{1,\tau}\pi_t^{1,\tau} + (1-\lambda)X_t^{2,\tau}\pi_t^{2,\tau}}{X_t^{\lambda,\tau}}, \quad \ell_t^{\lambda,\tau} = \frac{\lambda X_t^{1,\tau}\ell_t^{1,\tau} + (1-\lambda)X_t^{2,\tau}\ell_t^{2,\tau}}{X_t^{\lambda,\tau}}.
\end{equation}
Since $u^{1,\tau}$ and $u^{2,\tau}$ are admissible, $X_t^{1,\tau} > 0$ and $X_t^{2,\tau} > 0$ almost surely. This guarantees $X_t^{\lambda,\tau} > 0$, rendering $(\pi_t^{\lambda,\tau}, \ell_t^{\lambda,\tau})$ admissible, and $X^{\lambda,\tau}$ is the corresponding wealth process, with the consumptions rate
\begin{equation}
c_t^{\lambda,\tau} = \ell_t^{\lambda,\tau} X_t^{\lambda,\tau} = \lambda c_t^{1,\tau} + (1-\lambda)c_t^{2,\tau}.
\end{equation}
Since the power utility function $U(c)$ is strictly concave in consumption,
\begin{align*}
v^\tau(x_\lambda, y) &= \mathbb{E}\left[ \int_0^\infty e^{-\beta t} U(c_t^{\lambda,\tau}) dt \ \middle|\ \mathcal{F}_\tau^B \right] \\
&> \lambda \mathbb{E}\left[ \int_0^\infty e^{-\beta t} U(c_t^{1,\tau}) dt \ \middle|\ \mathcal{F}_\tau^B \right] + (1-\lambda)\mathbb{E}\left[ \int_0^\infty e^{-\beta t} U(c_t^{2,\tau}) dt \ \middle|\ \mathcal{F}_\tau^B \right]\\
=& \lambda v^\tau(x_1, y) + (1-\lambda)v^\tau(x_2, y).
\end{align*}
Thus $v^\tau$ is strictly concave in $x$.
\end{proof}
\subsubsection{Numerical Implementation}
We first consider one-dimensional Heston model, with constant risk free rate $r$, and 
\[
\begin{aligned}
\frac{dS_t}{S_t}
    &= (r+\Lambda Y_t)\,dt+\sqrt{Y_t}\,dZ_t,\\
dY_t
    &= \kappa(\theta-Y_t)\,dt+\nu\sqrt{Y_t}\,dW_t,
\end{aligned}
\]
where $d\langle Z,W\rangle_t=\rho\,dt$. Here, $Y_t$ represents
the instantaneous variance of the risky asset, and $\Lambda$
denotes the market risk-premium. We adopt model parameters used in
\citet{gutekunst2026optimalinvestmentconsumptionstochastic},
originally taken from \citet{GW2020}, and initialize the iteration with $\pi_0(y)=0.4$ and $\ell_0(y)=0.01$.
\begin{align*}
\begin{gathered}
\gamma=2,\qquad \beta=0.02,\qquad
r=0.013,\qquad \Lambda=1.66,\\
\kappa=0.088,\qquad \theta=0.035,\qquad
\nu=0.031,\qquad \rho=-0.84.
\end{gathered}
\end{align*}

With $\lambda(\tau) = 10^{-6}e^{-0.03\tau}$, Figure \ref{fig:consumption bound and convergence} shows that the consumption strategy lies between the approximation bound provided in \citet{GW2020} and coincide with the results in \citet{gutekunst2026optimalinvestmentconsumptionstochastic} after 800 iterations. The right panel shows that the difference error between neighboring rounds decreases exponentially on a logarithmic scale. For Heston model \(\tau \in [ 0,800]\) with step size \(\Delta \tau = 0.1\) and 8,000 total steps.
\begin{figure}[H]
\FIGURE
{\includegraphics[width=\textwidth]{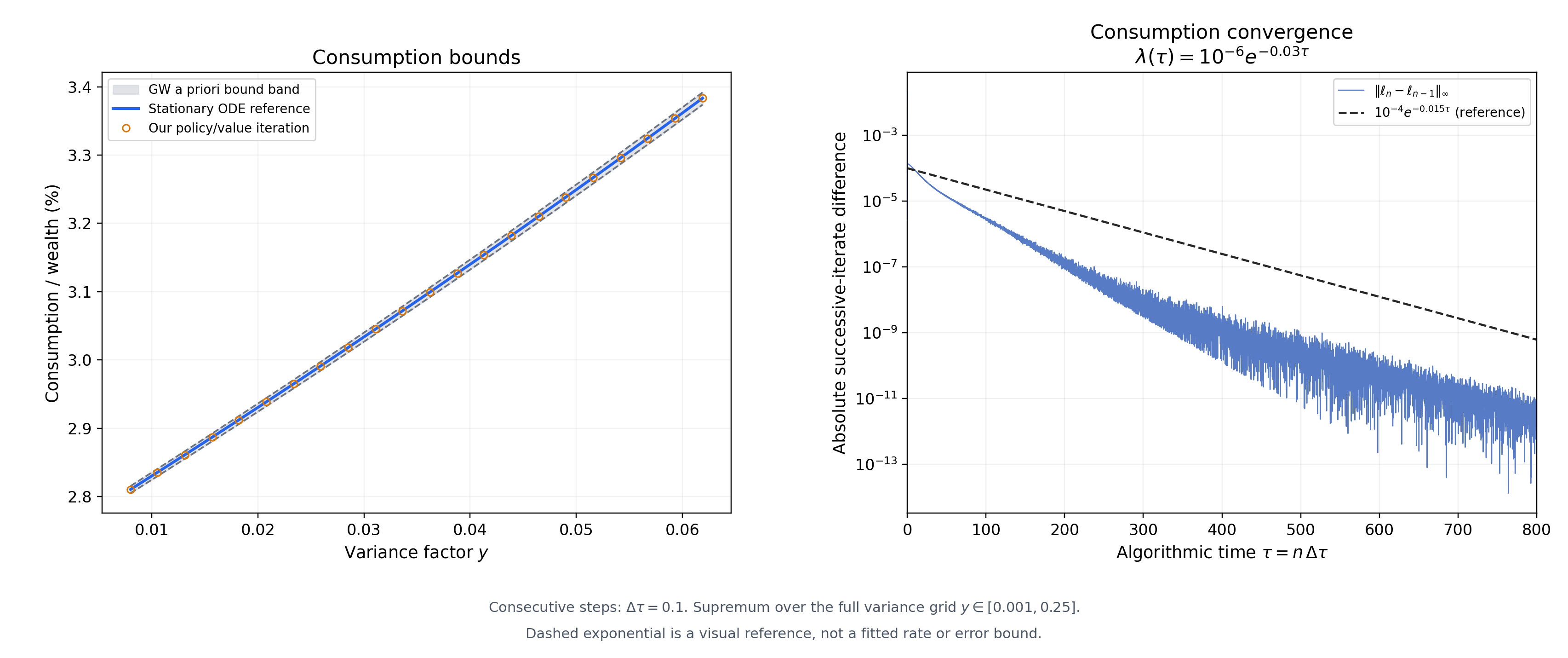}}
{Consumption and its convergence:
    The final iterate at $\tau=800$ (orange markers) lies within
    the bounds of \citet{GW2020} (gray band) at every checked
    grid point in $[0.008,0.062]$ and closely matches the benchmark from
\citet{gutekunst2026optimalinvestmentconsumptionstochastic}
    (blue curve).\label{fig:consumption bound and convergence}}
{}
\end{figure}

Figure \ref{fig: policy convergence} shows that the approximation for both the optimal control and the value function are accurate already after 200 iterations.
\begin{figure}[htbp]
\FIGURE
{\includegraphics[width=\textwidth]{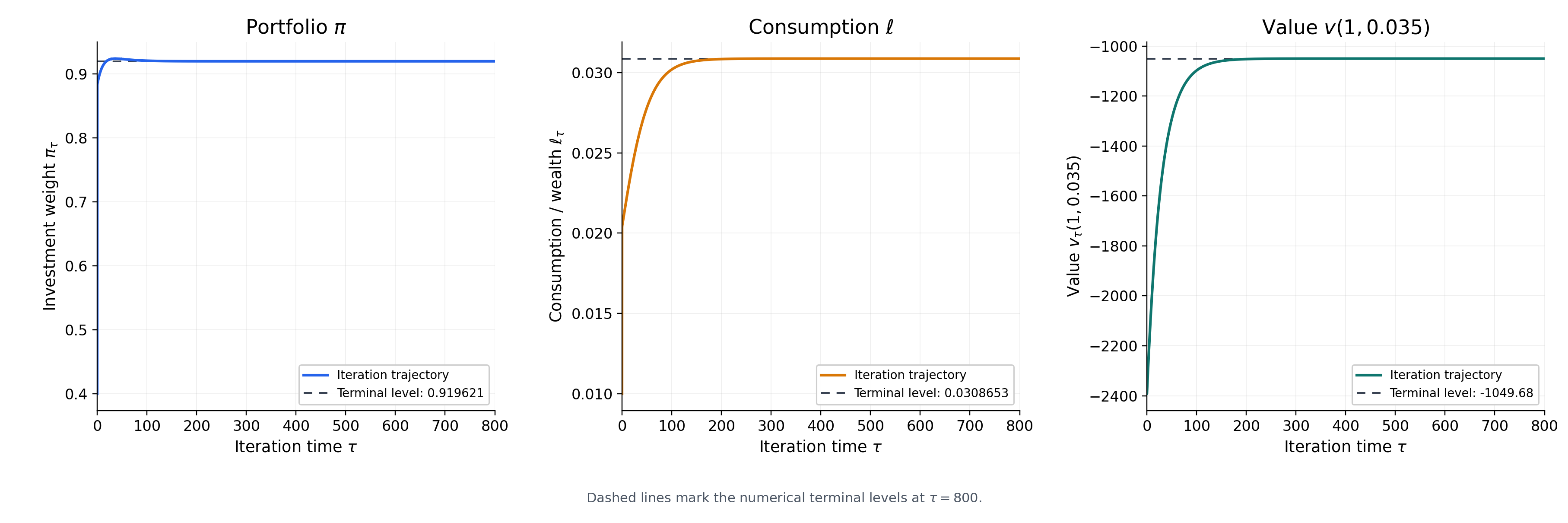}}
{Convergence to the optimum  at \(y = 0.035\) and \(x=1.\)\label{fig: policy convergence}}
{}
\end{figure}

\textbf{Two-factor Model.} We further investigate a model with two stochastic factors, for which the approximation methods in \citet{gutekunst2026optimalinvestmentconsumptionstochastic} does not apply. The factors \(Y_t^1\) and \(Y^2_t\) represent the market risk and (real) short rate. The market dynamics are 
\begin{align*}
    \frac{dS_t^0 }{S_t^0} =  Y_t^2 dt, \quad \frac{dS_t^1}{S_t^1}= (Y_t^2 + \sigma Y_t^1) dt + \sigma dZ_t,
\end{align*}
where the factors follow Ornstin-Uhlenbeck dynamics 
\begin{align*}
    dY_t^i= k_{i}(\theta_i - Y_t^i) dt + a_i dW_t^i, \quad i = 1,2.
\end{align*}
The Brownian motion satisfy $d\langle W^1, W^2 \rangle_t = 0$, and $d\langle Z, W^i\rangle_t = \rho_i dt$,for $i = 1,2$, and we choose 
\begin{align*}
\gamma = 2,\qquad \beta = 0.08,\qquad \sigma = 0.2, \qquad (\kappa_1,\kappa_2) = (0.3,0.2),\\
(\theta_1,\theta_2) = (0.35,0.04),\qquad (a_1,a_2) = (0.08,0.01),\qquad
(\rho_1,\rho_2) = (-0.8,-0.3).
\end{align*}
The iteration starts from \(\pi_0(y) = 0.5, \ell_0 (y) = 0.05.\) With $\lambda(\tau) = 0.05 e^{-0.1\tau}$, and step size \(\Delta \tau = 0.05\) on \(\tau \in [0,400]\). Since there is no benchmark solution to compare with in this two dimensional model, we assess the convergence through the successive differences of the policy and value, which decay exponentially in the right panel of Figure \ref{fig:policy-value-two-factors-convergence} on a logarithmic scale, for the approximation at $x=1$.  

\begin{figure}[H]
\FIGURE
{\includegraphics[width=\textwidth]{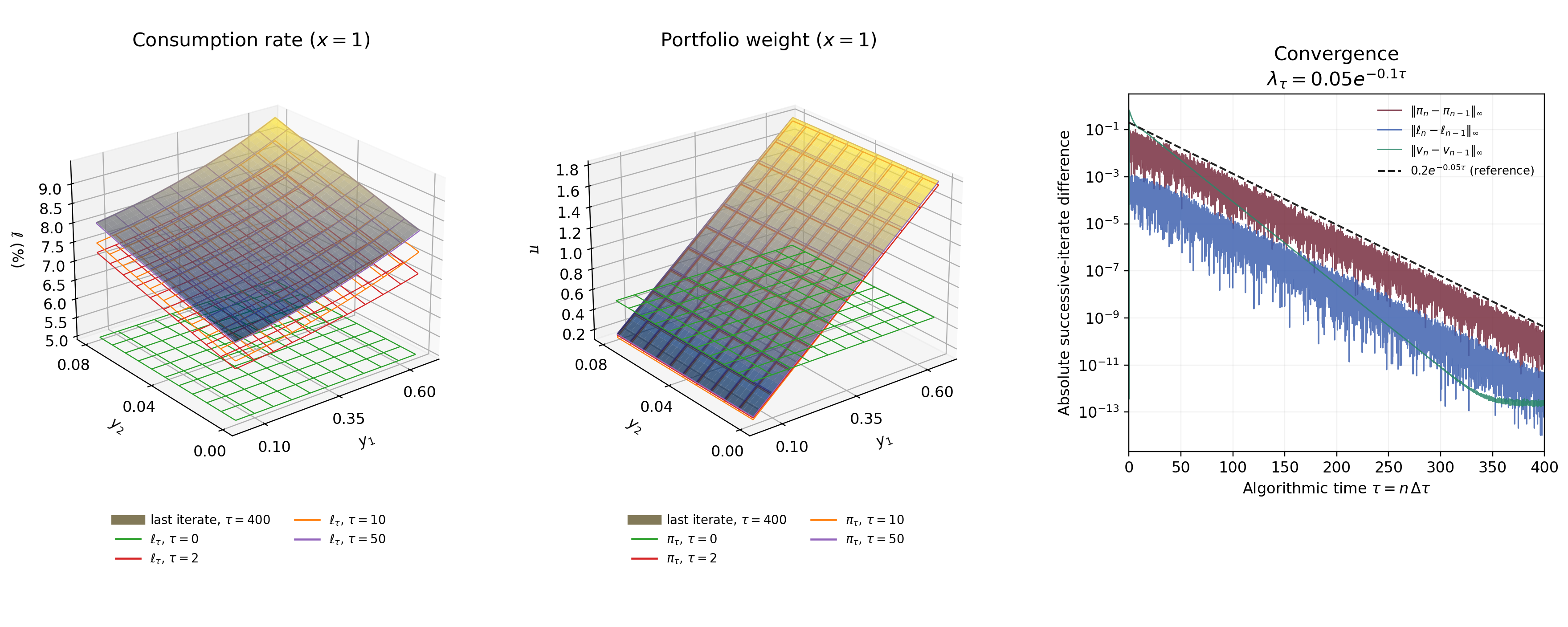}}
{Policy-value iteration with two stochastic factors. Left and middle figures: consumption rate and portfolio weight at iteration times \(\tau = 0,2, 10, 50, \) with shaded surfaces showing the final iterates at \(\tau = 400\). Right: successive-iterate differences. The dash line \(0.2 e^{-0.05 \tau}\) an exponential decay as a reference.\label{fig:policy-value-two-factors-convergence}}
{}
\end{figure}

The bounds for the optimal consumption from \citet{GW2020} provide another reference for our solution. At \(\tau = 61,\) the consumption rate lies between the bounds for the first time in the iteation at every grid point shown in the left panel of Figure \ref{fig:two-factor-bounds} (in order to distinguish between our approximate solution and the narrow bounds, we plot the distance from the upper bound, so that that the upperbound is represented as a constant $0$). While the central and right panel shows that from $\tau=61$ to the last iterate at $\tau=400$, the value continues to improve.
\begin{figure}[H]
\FIGURE
{\subfloat[Consumption rate at $\tau=61$.\label{fig:two-factor-bounds-consumption}]{%
    \includegraphics[width=0.49\textwidth]{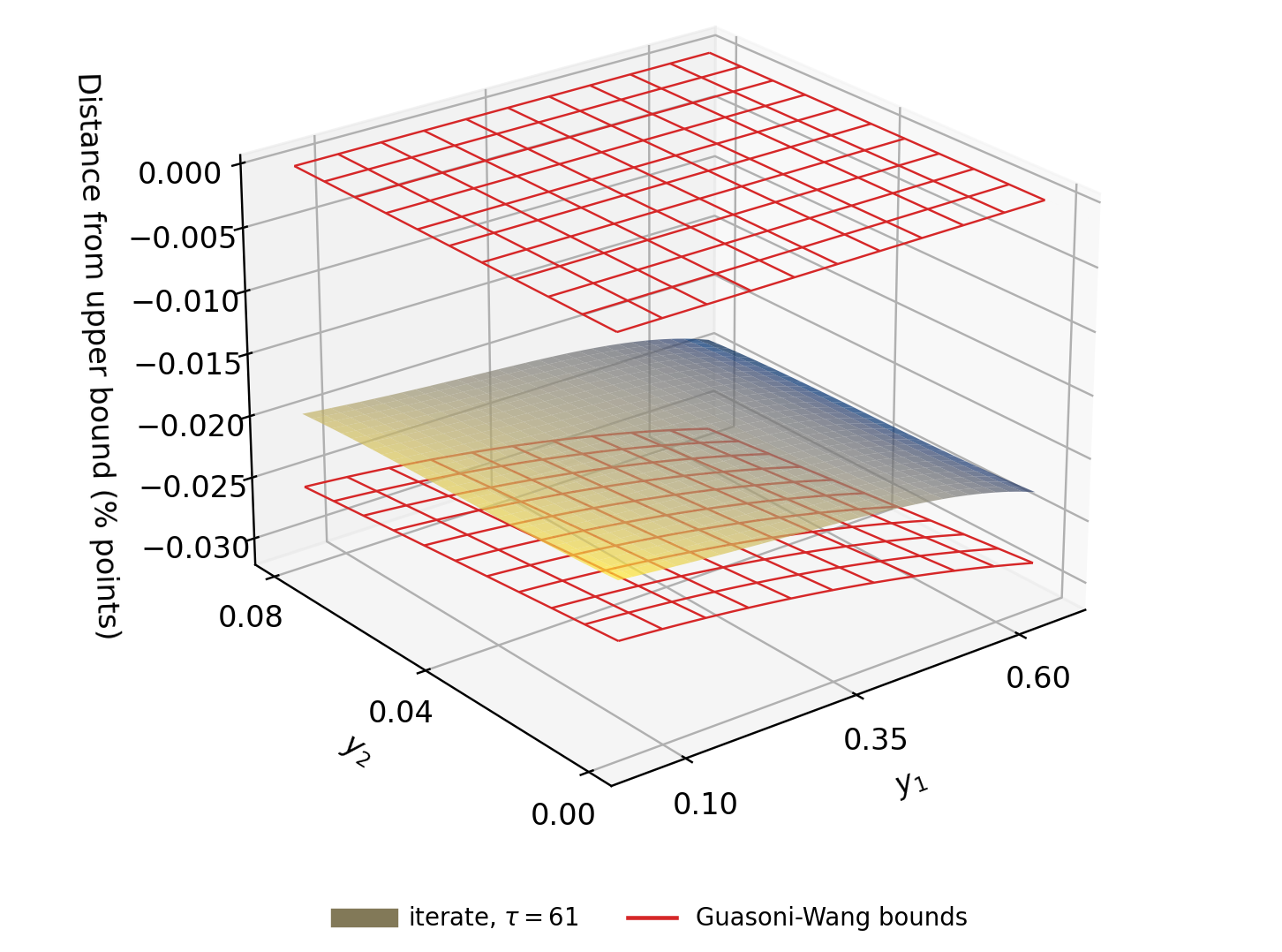}}\hfill
  \subfloat[Value at $\tau=61$ and $\tau=400$.\label{fig:two-factor-bounds-value}]{%
    \includegraphics[width=0.49\textwidth]{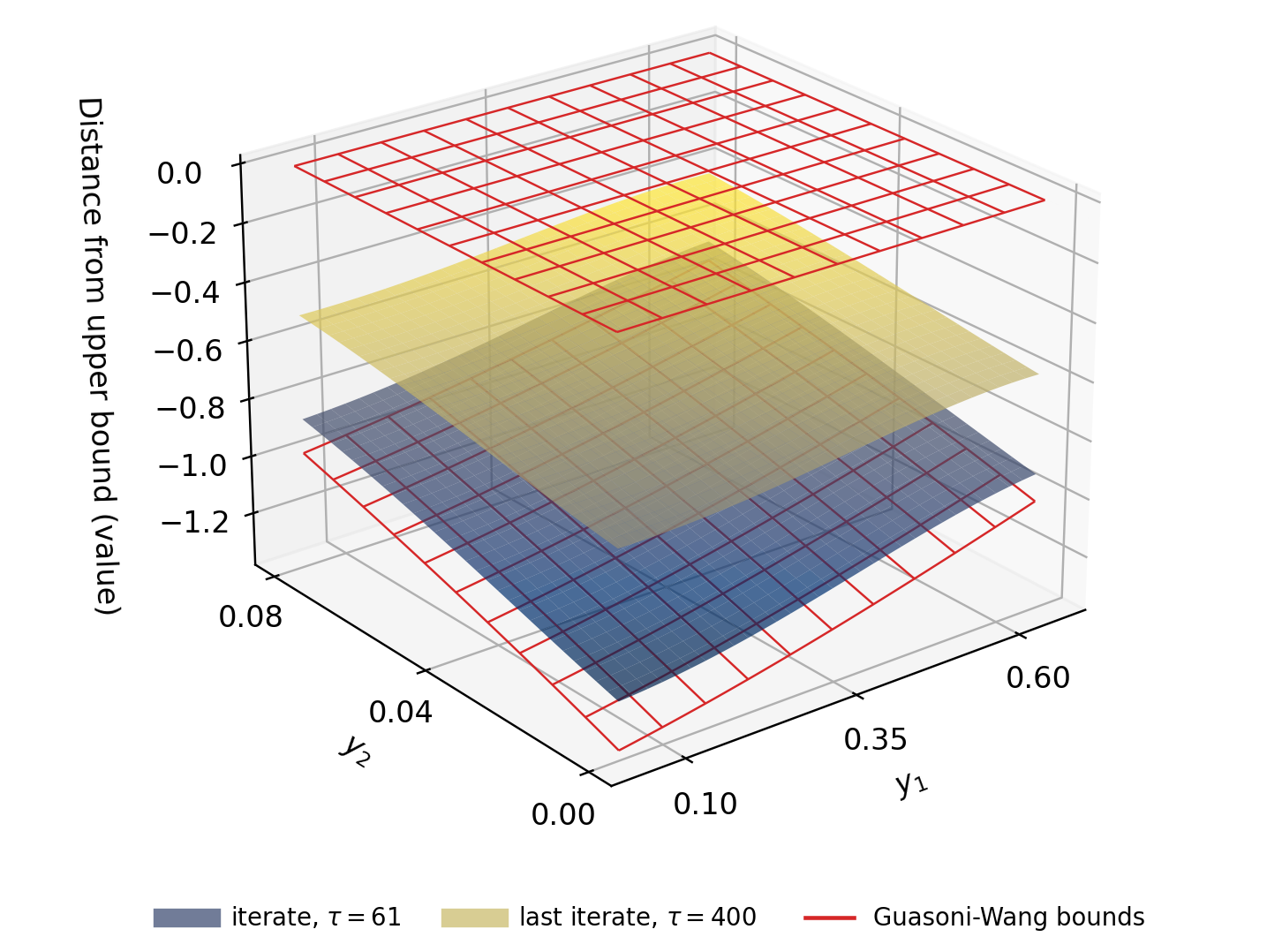}}}
{Distance from the upper bound at wealth $x=1$, for (a) the consumption rate and (b) the value. The red grids are the lower and upper bounds obtained from \cite{GW2020}. The iterates lie between the bounds at every grid point shown.\label{fig:two-factor-bounds}}
{}
\end{figure}

\bibliographystyle{agsm}
\bibliography{ref}

\end{document}